\documentclass[11pt,leqno]{amsart}
\usepackage{amsthm,amsfonts,amssymb,amsmath,oldgerm}
\numberwithin{equation}{section}
\usepackage{graphicx}
\usepackage{fullpage}

\renewcommand\d{\partial}
\renewcommand\a{\alpha}

\def\eps{\varepsilon }

\renewcommand\d{\partial}
\renewcommand\a{\alpha}

\newcommand\R{\mathbb R}
\newcommand\C{\mathbb C}

\def\eps{\varepsilon}

\newcommand\kernel{\hbox{\rm Ker}}

\newcommand\br{\begin{remark}}
\newcommand\er{\end{remark}}
\newcommand\bp{\begin{pmatrix}}
\newcommand\ep{\end{pmatrix}}
\newcommand{\be}{\begin{equation}}
\newcommand{\ee}{\end{equation}}
\newcommand\ba{\begin{equation}\begin{aligned}}
\newcommand\ea{\end{aligned}\end{equation}}

\newcommand{\bap}{\begin{app}}
\newcommand{\eap}{\end{app}}
\newcommand{\begs}{\begin{exams}}
\newcommand{\eegs}{\end{exams}}
\newcommand{\beg}{\begin{example}}
\newcommand{\eeg}{\end{exaplem}}
\newcommand{\bpr}{\begin{proposition}}
\newcommand{\epr}{\end{proposition}}
\newcommand{\bt}{\begin{theorem}}
\newcommand{\et}{\end{theorem}}
\newcommand{\bc}{\begin{corollary}}
\newcommand{\ec}{\end{corollary}}
\newcommand{\bl}{\begin{lemma}}
\newcommand{\el}{\end{lemma}}
\newcommand{\bd}{\begin{definition}}
\newcommand{\ed}{\end{definition}}
\newcommand{\brs}{\begin{remarks}}
\newcommand{\ers}{\end{remarks}}

\newcommand{\D }{\mathcal{D}}

\newcommand{\ZZ}{{\mathbb Z}}

\newcommand{\Id}{{\rm Id }}
\newcommand{\Range}{{\rm Range }}

\newcommand{\diag}{{\rm diag }}
\newcommand{\blockdiag}{{\rm blockdiag }}
\newcommand{\Span}{{\rm Span }}

\newcommand{\sgn}{\text{\rm sgn}}
\newtheorem{theorem}{Theorem}[section]
\newtheorem{proposition}[theorem]{Proposition}
\newtheorem{corollary}[theorem]{Corollary}
\newtheorem{lemma}[theorem]{Lemma}

\theoremstyle{remark}
\newtheorem{remark}[theorem]{Remark}
\theoremstyle{remarks}
\newtheorem{remarks}[theorem]{Remarks}
\theoremstyle{definition}
\newtheorem{definition}[theorem]{Definition}

\newtheorem{example}[theorem]{Example}

\newcommand\cD{{\mathcal  D}}

\newcommand\cH{{\mathcal  H}}

\newcommand\cS{{\mathcal S}}

\newcommand{\beq}{\begin{equation}}
\newcommand{\eeq}{\end{equation}}

\font\tenronde=rsfs10
\font\sevenronde=rsfs7
\font\fiveronde=rsfs5
\newfam\rondefam
\scriptscriptfont\rondefam=\fiveronde
\textfont\rondefam=\tenronde
\scriptfont\rondefam=\sevenronde

\newcommand{\re}{\mathrm{Re\,}}

\newcommand{\range}{\mathrm{range}}

\newtheorem{theo}{Theorem }[section]

\newtheorem{exams}[theo]{Examples}

\numberwithin{equation}{section}

\title{Pseudodifferential damping estimates and stability of relaxation shocks}

\author{Kevin Zumbrun}
\address{Indiana University, Bloomington, IN 47405}
\email{kzumbrun@indiana.edu}
\thanks{Research of K.Z. was partially supported
under NSF grant no. DMS-0300487}
\begin{document}

\begin{abstract}
A bottleneck in the theory of large-amplitude and multi-D viscous and relaxation 
shock stability is the development of nonlinear damping estimates controlling 
higher by lower derivatives.  These have traditionally proceeded from 
time-evolution bounds based on Friedrichs symmetric and
Kawashima or Goodman type energy estimates.  Here, we propose an alternative 
program based on frequency-dependent pseudodifferential time-space damping 
estimates in the spirit of Kreiss.  These are seen to be equivalent in the 
linear case to high-frequency spectral stability, and, 
just as for the constant-coefficient analysis of Kreiss, sharp in a pointwise, 
fixed-frequency, sense.  This point of view leads to a number of simplifications 
and extensions using both already-existing analysis and new results developed here.
In particular, we essentially resolve fully existence of damping estimates for 
smooth, noncharacteristic relaxation waves in the small-amplitude case for multi-D 
and in the large-amplitude case for 1-D, showing that both reduce to satisfaction 
of a standard linear-algebraic dissipativity condition at the endstates of the 
shock, related to high-frequency stability of constant solutions.
For large-amplitude multi-D shocks, we point out a key new difficulty 
in the form of Airy-type 
\emph{turning points} analogous to Kreiss' glancing points in the 
constant-coefficient case, and introduce a variable-coefficient 
adaptation of Kreiss' symmetrizers allowing for the first time the treatment of
large-amplitude 
	noncharacteristic relaxation shocks in multi-D.
\end{abstract}

\date{\today}
\maketitle


\tableofcontents


  \section{Introduction}\label{s:intro}
  In this paper, we propose a new ``pseudodifferential type'' damping estimate for the stability analysis of
  1- and multi-D stability of planar relaxation waves, i.e., 
traveling-wave solutions $\bar w(x_1)$ of {\it relaxation systems} \cite{Li,Bre,Da}
\be\label{blaw}
\d_t f_0(w) + \sum_{j=1}^d \d_{x_j} f_j(w)=r(w),
\ee
a type of {\it hyperbolic balance law} arising in non-equilibrium mechanics, in 
which the force term $r$ vanishes on an attracting equilibrium set $\{r(w)=0\}$. 
Our studies are particularly motivated by shallow-water flow, for which the understanding of {\it hydraulic
shocks} or ``bores'' and anomalous large-amplitude discontinuous {\it roll wave} solutions are important in 
design of coastal, dam, and canal structures \cite{YZ,YZ2,JNRYZ,RZ,RZ2}.
However, the issues, and methods proposed, are common to general systems \eqref{blaw}; indeed, they 
appear useful also for stability of planar viscous shock fronts, or traveling-wave solutions of second-order
parabolic and mixed hyperbolic-parabolic systems.

A damping estimate, as introduced in \cite{Z1,Z2,MaZ1}, is an energy estimate
\be\label{damp}
\d_t \mathcal{E}(v(t))\leq -\eta \mathcal{E}(v(t)) + C\|v(t)\|_{L^2_\alpha}^2,
\qquad \eta>0,
\ee
for an evolution system in $v$, 
where $L^2_\alpha$ denotes a weighted $L^2$ space and $\mathcal{E}(s)\sim \|w\|_{ H^s_\alpha}^2$
is equivalent to a higher-derivative Sobolev space with the same weight.
In its simplest form $s=2$, $\alpha \equiv 1$, this reduces to the classic ``Kawashima estimate'' \cite{Ka,KaS,Ze} used to study stability of constant or near-constant solutions.  The use of nontrivial weights was introduced in
\cite{MaZ4,Z1} to treat stability of large-amplitude, asymptotically constant waves
and profoundly generalized in \cite{MaZ2,RZ,YZ}.

Such an estimate prohibits singularity formation, allowing analysis in standard Sobolev spaces.
Indeed, it encodes exponential slaving of the $H^s_\alpha$ norm to $L^2_\alpha$, effectively controlling higher
derivative norms by lower ones.
This is extraordinarily helpful in situations of maximal or delicate regularity, controlling apparent derivative
loss from linear estimates to close a nonlinear iteration \cite{Z1,Z2}.
More \cite{Z2}, such estimates typically 
yield also sharp high-frequency resolvent bounds, thus reducing the problems of linearized estimates and
nonlinear stability to obtaining sharp low-frequency resolvent estimates and converting them to sharp estimates
on the temporal solution operator.
Some extreme examples are the damping estimates used in \cite{MaZ2,YZ} to treat 1-D stability of large-amplitude
relaxation fronts, including fronts with ``subshocks'', or discontinuities in the shock profile,
and the 1-D estimates derived in \cite{RZ2} for discontinuous periodic roll wave solutions of the Saint
Venant equations for inclined shallow-water flow.

Our starting point is the theme introduced in \cite{Z2}, and generalized greatly in \cite{RZ,RZ2}, that 
{\it nonlinear damping estimates} for a nonlinear perturbation $v$ of a steady solution $\bar w$ 
of \eqref{blaw} are linked to {\it high-frequency resolvent estimates}
\be\label{hfres}
|(\lambda- L)^{-1}|_{\hat H^s}\leq \frac{C}{\Re \lambda -\gamma_*}, 
\; \hbox{$\gamma_*<0$, for all $|\xi, \tau|$ sufficiently large}
\ee
on the linearized operator $L$ about the wave $\bar w$, such as have been central in the linear 
estimation of behavior for viscous shock and relaxation waves \cite{MaZ1,MaZ2,MaZ4,Z2,Z3}.
In particular, when \eqref{damp} holds via quadratic energy estimate, 
we can establish a bound \eqref{hfres} using a linearized version of 
the same estimates, a sort of ``restricted Lumer--Phillips theorem'' \cite{Z2}.
And, when \eqref{hfres} holds, by any type of estimate, we can often recover a nonlinear version \eqref{damp}
obtained by quadratic energy estimate \cite{RZ,RZ2}.
In some cases, e.g., \cite{Z1,YZ}, these properties are imposed ``statically'' by structure of the {\it equations};
in others, as in \cite{RZ2}, they are imposed ``dynamically'' by structure of the {\it solution}.
However, in general, condition \eqref{damp} appears to be stronger than condition \eqref{hfres}
Thus, there is a gap between the estimates that we essentially require for the linear theory, and
those that would allow us to complete a nonlinear iteration.

The goal of this work is to close this gap by identifying a more general type of nonlinear damping estimate
that is truly equivalent to high-frequency linear damping, but still suitable to carry out a standard nonlinear
iteration such has been developed in \cite{MaZ1,Z2,Z3}.
Specifically, for maximum flexibility, keeping in mind the example of the finite-time
Kreiss-Majda analysis of conservative shock waves \cite{K,MZ3,MZ4} and the fact that at least one
of our envisioned applications involves global-in-time analysis of relaxation fronts involving subshock
discontinuities,
we seek {\it pseudodifferential}, or frequency-dependent damping estimates.
These can then be verified by Kreiss symmetrizer estimates, or, alternatively, by establishing
uniform smooth exponential dichotomies with respect to frequencies: for example, 
by classical WKB expansion as typically carried out in high-frequency analysis of the
associated linear stability problem \cite{JNRYZ,RZ,RZ2,YZ,YZ2}.

\subsection{Main results}\label{s:results}
The development of suitable damping estimates is a current bottleneck in the stability theory for relaxation
shocks, with a number of situations in which optimal linear bounds have already been established, with 
nonlinear stability waiting only for a damping estimate and associated control of regularity: in particular
large-amplitude multi-D stability of relaxation shocks \cite{YZ2}, and small-amplitude multi-D
stability without the symmetrizability assumptions of \cite{Kw,KwZ}.

Here, we address this bottleneck with the addition of new tools,
for simplicity restricting mainly
to the case of {\it smooth waves}, for which we 
need consider only interior structure without the added complication of 
boundary/transmission conditions.
We first propose in place of \eqref{damp} a pseudodifferential analog \eqref{pdamp} 
that can be shown to be linearly equivalent to \eqref{hfres}.
{\bf See Proposition \ref{lineqprop}.}
Next, we show that existence of a multi-dimensional 
Kreiss symmetrizer implies a time-integrated version of the classical nonlinear 
damping estimate, and discuss relations to exponential dichotomies and 
WKB expansion. 
{\bf See Propositions \ref{nlprop} and \ref{lyapprop}.}

In the small-amplitude case, under mild additional assumptions
of hyperbolicity and constant multiplicity (conditions \eqref{A2}-\eqref{A3} below),
we show that nonlinear pseudodifferential damping is equivalent to
the restriction to high-frequencies of the linear dissipativity 
condition of Kawashima \cite{Ka,KaS} (\eqref{chf} below) needed for 
diffusive stability of nearby constant solutions, 
at a sufficiently nearby reference state, e.g., at either endstate of the shock.
{\bf See Theorem \ref{constequivcor}}.
In the large-amplitude 1-D case, by the use of exponential dichotomies,
we show under the same hypothesese that
pseudodifferential damping is equivalent to the 1-D version of
the same high-frequency linear dissipativity condition, but imposed at 
{\it both endstates} of the shock.
{\bf See Theorem \ref{1dwkb}.}

In the large-amplitude multi-D case, we point out the new issue of Airy-type
turning points, and an associated distinction between exponential dichotomies
and symmetrizers, namely, that existence of dichotomies is strictly stronger
than existence of symmetrizers, so that symmetrizer construction is the tool
of more general application.
{\bf See Propositions \ref{nodice} and \ref{yessymm}.}
Combining exponential dichotomies with Kreiss-type constant-coefficient
symmetrizers, we carry out a symmetrizer construction in a generic case
involving single Airy-type turning points. {\bf See Theorem \ref{basicthm}.}
{\it This makes possible for the first time the treatment of nonlinear stability
of large-amplitude multi-D relaxation shocks.}
In particular, our results apply to arbitrary-amplitude 
smooth hydraulic shock waves of the Saint Venant equations,
as studied in \cite{YZ2}.
We show in Section \ref{s:nests} that the time-integrated nonlinear damping estimate resulting from our pseudodifferential damping condition
serves equally well as the classical one to close a nonlinear iteration and 
yield a nonlinear stability result.  {\bf See Theorem \ref{smooththm}.} 

Along the way, we discuss in Sections \ref{lineqdisc}, \ref{s:discdich}, 
\ref{s:disc1D}, \ref{s:gendisc}, and \ref{s:absorb} 
the changes needed for extension to the discontinuous case,
in particular sketching the treatment in 1- and multi-D 
of arbitrary-amplitude shock profiles with a single subshock discontinuity, 
of the types arising for discontinuous hydraulic shock waves of
the Saint Venant equations \cite{YZ,FRYZ,YZ2}.
{\bf See Proposition \ref{lyappropdisc},
Corollary \ref{1dwkbdisc},} and the treatment in {\bf Theorem \ref{basicthmdisc}}
of a multi-d turning point case scenario noted by Erpenbeck \cite{Er4,LWZ,Z9} in the case of detonations that as shown in \cite{YZ2} is relevant
for discontinuous hydraulic shock profiles of
the Saint Venant equations of inclined shallow-water flow.

Finally, we discuss applications and larger perspectives in 
Sections \ref{s:appl}-\ref{s:disc}, in particular pointing out as
an important open problem the complete treatment of nonlinear stability
of large-amplitude discontinuous profiles in multi-D,
building on the nonlinear damping estimates established here.

\medskip

{\bf Acknowledgement:} Thanks to M. Williams and D. Lannes for helpful discussions
regarding quantitative pseudodifferential estimates, and to M. Williams for reading an early version
of this manuscript and making several helpful suggestions and comments. 
Our ideas on this subject cannot be separated from those of L.-M. Rodrigues, with whom we have
had continuing discussion and collaboration on these and related issues for over 15 years, and indeed
we claim no novelty other than the placing of the various ideas in a coherent and general context that
may be convenient for a general audience.
We note the appearance of related observations in the treatment of stability of constant solutions
by M. Sroczinsky in \cite{Sr}, and in the general treatment of stability of discontinuous
solutions by G. Faye and L.-M. Rodrigues in \cite{FR1,FR2}.

\medskip
{\bf Dedication:} {\it to my mother, Ann Zumbrun, for her company and example
during this project.}

\section{Pseudodifferential damping and high-frequency stability}\label{s:pseudo}
Let $\bar w$ be a planar traveling wave of \eqref{blaw} taken without loss of generality
(by change of coordinates) to be steady and depending only on $x_1$, i.e., a solution
$$
w(x,t)= \bar w(x_1).
$$
Consider following the approach of \cite{MaZ1,Z3} a nearby wave 
$w(x,t)$ and perturbation $v$ defined by
$$
w(x_1+\psi(x_2,\dots,x_d,x_2,t), x_2, \dots, x_d, ,t)= \bar w(x) +v(x,t),
$$
where $v$ is a nonlinear residual and $\psi$
is a possible small phase modulation depending on $v$ and $\bar w$.

Then, the nonlinear perturbation equations governing $v$ take the form 
\be\label{nonlin}
 v_t + \sum_{j=1}^d A_j \partial_{x_j} v   + E v=  f,
 \ee
 where 
 $$
 A_1=(df_1/dw)(\bar w +v)+ \partial_{t} \psi \, \Id +
 \sum_{j=2}^d \partial_{x_j}\psi \,(df_j/dw)(\bar w+v), 
 $$
 $A_j=(df_j/dw)(\bar w +v)$ for $j=2,\dots,d$,
 $Ev=-(dr/dw)(\bar w)v + (d^2 f_1/dw^2)(v,\partial_{x_1}\bar w)$, and 
 $f$ is a quadratic order remainder term in $v$ and derivatives of $\psi$ multiplied by factor $\partial_{x_1}\bar w$, depending also on $\bar w$ and its derivatives, 
 but not on $\psi$ or derivatives of $v$.

The incorporation of a phase-modulation $\psi$ is necessary in 1-D or 
discontinuous cases, but typically not needed for smooth waves in multi-D \cite{Z3}.
For a standard multi-D analysis of smooth waves-- our main object here--
the reader may consider for simplicity $\psi(x,t)\equiv 0$.
 %

 \subsection{Classical damping as a priori bound}\label{s:classical}
 Consider now the partially linearized equation \eqref{nonlin} with general righthand side $f$.
 The classical damping estimate of \cite{MaZ1,Z2,Z3}, broken down to its essential components, is an 
{\it a priori estimate}
 \be\label{essdamp}
\d_t \mathcal{E}(v(t))\leq -\eta \mathcal{E}(v(t)) + C(\|v(t)\|_{L^2_\alpha}^2 +\|f(t)\|_{H^s_\alpha}^2),
\qquad \eta>0,
 \ee
 for \eqref{nonlin}, involving both $v$ and the unknown $f$, where $\mathcal{E}(v)\sim \|v\|_{H^s_\alpha}^2$.

 When $f$ is a quadratic order remainder as described just above, its $v$-component may be absorbed in the lefthand
 side of \eqref{essdamp}, leaving in the case that $\psi\equiv 0$ the estimate \eqref{damp} of the
 introduction, and in the general case
 $ \d_t \mathcal{E}(v(t))\leq -\eta \mathcal{E}(v(t)) 
+ C(\|v(t)\|_{L^2_\alpha}^2 +\|\partial_{x,t}\psi(t)\partial_{x_1}\bar w\|_{H^s_\alpha}^2):$
in the 1-D case, with $\psi=\psi(t)$ independent of $x$, simply
$\d_t \mathcal{E}(v(t))\leq -\eta \mathcal{E}(v(t)) + C(\|v(t)\|_{L^2_\alpha}^2 +|\dot \psi(t)|^2)$
as in \cite{MaZ1,Z3}.

 \subsubsection{Time-integrated version}\label{s:timeint}
 From \eqref{essdamp}, one obtains by Gronwall inequality the integral version
 \be\label{idamp}
 \|v(T)\|_{H^s_\alpha}^2\leq Ce^{-\eta T} \|v(0)\|_{H^s_\alpha}^2
 + C\int_0^T e^{-\eta (T-t)}\big( \|v(t)\|_{L^2_\alpha}^2 + \|f(t)\|_{H^s_\alpha}^2 \big)\, dt
 \ee
 used in the nonlinear iteration arguments of \cite{Z2,Z3} after absorbing $f$ terms 
 as described just above- this can be done
 either before time-integration as above, or, more generally, in \eqref{idamp} itself.

 Bounds \eqref{essdamp}
 are obtained by weighted energy estimates,
 with $\mathcal{E}$ a weighted inner product, under additional symmetry and structure assumptions on coefficients
 $A_j$ and $E$, following the standard Friedrichs symmetric hyperbolic/Kawashima-Shizuto compensator approach, together with ``Goodman-type'' convective estimates taking advantage of noncharacteristicity \cite{MaZ1,MaZ2,Z2,Z3}.

 \subsection{Paralinearization}\label{s:paralin}
 To remove extraneous symmetry assumptions and connect directly to the linearized resolvent equation,
 we perform a pseudodifferential analysis, replacing $\partial_{x_j}$, 
 $j=2, \dots, d$ and $\partial_t$ by their symbols $i\eta_j$ and $i\tau$.
 See \cite[Appendix B]{MZ4} for a description of the relevant tools, phrased in the
 paradifferential calculus of Bony; different from the setting of \cite{MZ4}, there is no additional parameter
 here such as viscosity or relaxation constant, simplifying slightly the situation.
 This has the added advantage, in the case of planar traveling-wave profiles containing discontinuous subshocks, 
of putting things in a common framework with the finite-time stability analysis of discontinuous shock 
problems pioneered in \cite{K,M1,M2,Met3}, where frequency-dependent (in time)
estimates are seen to be necessary for multi-D stability estimates.

This yields, after a standard change of coordinate $v\to e^{\gamma t}v$ imposing an exponentially time-weighted norm,
a version of the nonlinear equation
\be\label{pres}
 \lambda v + A_1 \partial_{x_1} v + \sum_{j=2}^d i\xi_j A_j v + Ev= f
\ee
similar in form to the linearized resolvent equation after Fourier transform in $x_2, \dots, x_d)$,
with $\tau ,$ $\xi_j\in \R$  temporal and spatial frequencies, and $\lambda=\gamma +i\tau$.
In the case of a wave with discontinuous subshock, this becomes, rather
\be\label{dpres}
 \lambda v + A_1 \partial_{x_1} v + \sum_{j=2}^d i\xi_j A_j v + Ev= f
\, \hbox{\rm for $x_1>0$, and  $\Gamma v|_{x_1=0}=g$,}
\ee
where $\Gamma=\Gamma(w, \eta,\lambda)$ is a frequency-dependent boundary condition obtained from the linearized Rankine-Hugoniot jump conditions about $w$, with estimates for the subshock front perturbation $\psi$ recoverable from $v$.
For further details, see, e.g., \cite{Met3,YZ}, or \cite[Appendix]{BMZ}.


\subsection{Pseudodifferential damping}\label{s:dampest}
In the short-time theory of \cite{K,M1,M2,Met3}, 
the goal is, by frequency-dependent type symmetric energy estimates (see {\it Method of symmetrizers} below)
to obtain uniform ``semigroup type'' estimates 
\be\label{semigpest}
\|v\|_{\hat H^s}\leq \frac{C\|f\|_{\hat H^s}}{\Re \lambda -\gamma_*}, 
\; \hbox{\rm where $\|f\|_{\hat H^s}:= \|f\|_{H^s}+ (1+|\xi|)^s\|f\|_{L^2}$,}  
\ee
for \eqref{pres}, and analogous estimates \eqref{semigpestdisc}
for \eqref{dpres} involving also $g$.
Parseval's identity gives then 
\be\label{bpars}
\int_0^\infty e^{2 \gamma t}\|v(t)\|_{H^s}^2 dt 
\leq C_2 \int_0^\infty e^{2 \gamma t}\|f(t)\|_{H^s}^2 dt ,
\ee
for any $\gamma >\gamma_*$ and
$v$, $f$ vanishing for $t\leq 0$ and satisfying \eqref{nonlin}, and thus 
for $v$, $f$ vanishing also on $t\geq T$, 
\be\label{bparcor}
\int_0^T e^{2 \gamma (T- t}\|v(t)\|_{H^s}^2 dt
\leq C_2 \int_0^T e^{2 \gamma (T-t)}\|f(t)\|_{H^s}^2 dt,
\ee
similarly as would be obtained by Duhamel's principle via time-exponential $C_0$ semigroup estimates.

However, the tools used to obtain this of integration by parts/G\"arding inequality
and Parseval's identity are sufficiently robust to be reproduced in the {\it nonlinear setting} via 
corresponding pseudodifferential calculations.
For $v$, $f$ compactly supported in time, 
we need only establish this for a single $\gamma>\gamma_*$, and not all, analogous to but more flexible
than Pr\"uss' theorem \cite{Pr} for semigroups.

For this bounded-time theory, $\gamma_*$ may be taken positive, and as large as desired, allowing exponential
growth at bounded rate.  For our different purposes here, we require $\gamma_*$ stricty negative at high
frequencies, making the bound \eqref{semigpest} impossible at low frequencies for the situations we consider,
in which at best time-algebraic stability of perturbations is expected to occur.
Thus, analogous to \eqref{damp}, we seek a 
{\it pseudodifferential damping condition} for problem \eqref{pres} of
\be\label{pdamp}
\|v\|_{\hat H^s}\leq \frac{C(\|f\|_{\hat H^s}+ \|v\|_{L^2})}{\Re \lambda -\gamma_*} 
\; \hbox{for some $\gamma_*<0$.}
\ee


\subsection{Global a priori estimate}\label{s:gapriori}
From \eqref{pdamp} we obtain by Parseval's inequality
immediately a global-in-time a priori bound generalizing \eqref{bpars},
\be\label{gpars}
\int_{-\infty}^\infty e^{2 \gamma t}\|v (t)\|_{H^s}^2 dt 
\leq C \int_{-\infty}^\infty e^{2 \gamma t}(\|f(t)\|_{H^s}^2 + \|v(t)\|_{L^2}^2 dt ,
\ee
and similarly for spatially weighted norms.
In particular, for $v$, $f$ supported on $0\leq t\leq T$, we obtain similarly as in \eqref{bparcor}
\be\label{gparcor}
\int_0^T e^{2 \gamma (T-t}\|v(t)\|_{H^s}^2 dt 
\leq C_2 \int_0^T e^{2 \gamma (T-t)}\big(\|f(t)\|_{H^s}^2+ \|v(t)\|_{L^2} \big) dt,
\ee
which may be recognized as an integral form of the classical damping estimate \eqref{idamp} in
the case of zero initial data.
With small adaptations, as we shall show in Section \ref{s:nests}, this can substitute for the classical
one in closing the standard nonlinear iteration scheme of \cite{MaZ1,MaZ2,Z2,Z3,YZ}.

\subsection{Linear equivalence}\label{s:linequiv}
Considering \eqref{pres}--\eqref{dpres} as ODE in $x_1$, we make the standard assumption \cite{MaZ1},
valid in the physical setting of hydraulic Saint Venant shocks \cite{YZ} (but not for Saint Venant roll 
waves \cite{JNRYZ,RZ2}), that the principal term $A_x \partial_{x_1}$ is nonsingular, i.e.
\be\tag{A1}\label{A1}
\hbox{\rm $A_1$ is uniformly invertible.}
\ee
Then, we have readily for the linear resolvent equation the following equivalence.

\begin{proposition}\label{lineqprop}
For the linearized resolvent equation, \eqref{pres} with $A_j:=(df_j/dw)(\bar w))$,
under structural assumption \eqref{A1}, damping condition \eqref{pdamp} and high-frequency stability 
condition \eqref{hfres} are equivalent for common $\gamma_*$, $\gamma$ and all $s\geq 1$, 
with possibly different constants $C$.
\end{proposition}

\begin{proof}
	Taking the $L^2$ inner product of $A_1 \partial_{x_1}v$ against itself, and noting that all other terms
	in the resolvent equation are bounded multiples of $v$, we find that $\|v\|_{\hat H^1}\leq C\|v\|_{L^2}$
	for any bounded frequencies $\eta, \lambda$. Differentiating the equations and proceeding in the same
	way, we obtain $\|v\|_{\hat H^s}\leq C\|v\|_{L^2}$ for any $s$ such that $A,E$ are smooth enough,
	verifying condition \eqref{pdamp} for all bounded frequencies. 
	Thus, we have reduced the problem to showing equivalence for large frequencies.
	On the other hand, taking the $L^2$ inner product of $\lambda v$ against itself, and noting that
	all other terms in \eqref{pres} are bounded by $C(\|v\|_{\hat H^1}+\|f\|_{L^2})$, 
	we have for $|\lambda|$ sufficiently large that 
	$$
	\|v\|_{L^2}/(\|v\|_{H^1}+\|f\|_{L^2})\leq C/|\lambda|
	$$
	is arbitrarily small, hence the $L^2$ term in \eqref{pdamp} may be absorbed, leaving the apparenty
	stronger condition\eqref{hfres}.  For $|\eta|$ sufficiently large, the implication is still clearer.
\end{proof}

\br\label{weightrmk}
Essentially the same argument yields the result also for weighted norms $\hat H^s_\alpha$ and $L^2_\alpha$.
\er

{\bf Conclusion:} At least at the linear level, high-frequency stability for the resolvent equation
is under condition \eqref{A1} equivalent to satisfaction of the pseudodifferential damping condition \eqref{pres}, 
whereas it does not necessarily imply a (linear version of the) pointwise damping estimate \eqref{damp}.
In Section \ref{s:nlinequiv} below, we examine the passage from linearized to nonlinear damping estimates.

\br\label{parabrmk}
A similar equivalence holds evidently in the strictly parabolic case, or in any situation for which the
highest order derivative term has nonsingular coefficient.
\er

\br
Failure of condition \eqref{A1} does not imply failure of damping, nor linear equivalence,
but only of the argument above.
Given the important example of Saint Venant roll waves \cite{JNRYZ,RZ2}, it would be very interesting
to investigate this case further using other techniques.
\er

\subsubsection{The discontinuous case}\label{lineqdisc}
For the case \eqref{dpres} of a single discontinuity, the
analog of \eqref{semigpest} is
\be\label{semigpestdisc}
\|v\|_{\hat H^s}\leq \frac{C(\|f\|_{\hat H^s}+ |g|_{\hat H^s})}{\Re \lambda -\gamma_*}, 
\ee
where $\|f\|_{\hat H^s}:= \|f\|_{H^s}+ (1+|\xi|)^s\|f\|_{L^2}$ and
$|g|_{\hat H^s}:= (1+|\xi|)^s|g|_{L^2}$,
and the analog of \eqref{pdamp} is
\be\label{pdampdisc}
\|v\|_{\hat H^s}\leq \frac{C(\|f\|_{\hat H^s}+ |g|_{\hat H^s} +  \|v\|_{L^2})}{\Re \lambda -\gamma_*}, 
\ee
for some $\gamma_*<0$.
Evidently, the same (tautological) argument as in the proof of 
Proposition \ref{lineqprop} gives linear equivalence of these two conditions
as well, as it makes no use of boundary conditions.

\section{Nonlinear equivalence: the method of symmetrizers}\label{s:nlinequiv}
We now investigate under what circumstances linear damping \eqref{pdamp} may be extended to a nonlinear 
result.
In the general, large-amplitude (possibly discontinuous case), we
present a result generalizing the ``restricted Lumer-Phillips theorem'' noted in \cite{Z2} 
for the local damping estimate \eqref{damp}. 
Namely, we show that when high-frequency estimate \eqref{hfres} can be established by a more general,
frequency-dependent symmetric quadratic energy estimate of the type introduced by Kreiss \cite{K},
then the pseudodifferential damping estimate \eqref{pdamp} may be established by such an energy estimate
as well.
For the smooth, small-amplitude case, we prove the much stronger, and easier, result
that linearized  and nonlinear damping and high-frequency resolvent estimates are all equivalent, 
reducing to the corresponding properties for a constant solution, a linear algebraic condition on the
Fourier symbol in $x_1, \dots, x_d$.

\subsection{Symmetrizers}\label{s:symm}
We begin by repeating essentially verbatim the generalization of Kreiss' method to variable-coefficient 
problems in \cite[\S 2.1]{MZ3}, there stated for the discontinuous case.
In the simpler smooth case, the spatial domain is extended to $(-\infty,\infty)$ and \eqref{clS0} dropped.
Recall now the essence of the ``method of symmetrizers'' as it
applies to general boundary value problems
\begin{equation}
\label{eqabstr}
\partial _x u = G(x)  u + f \, , \quad \Gamma u (0) = 0\, .
\end{equation}
Here, $u$ and $f$ are functions on $[0, \infty[$ values in some Hilbert
space $\cH$, and $G (x) $ is a $C^1$ family of (possibly unbounded)
operators defined on $\cD$,  dense subspace of     $\cH$.

\smallskip

 {\bf A  symmetrizer}   is a family of  $C^1$  functions
$x \mapsto S  (x)$ with values in the space of
operators in $\cH$
such that there are $C_0$, $ \lambda > 0$,  $\delta >
0$ and $C_1$ such that
\begin{eqnarray}
\label{symS0} && \forall x\, ,    \quad S(x) = S(x)^* \quad   {\rm  and }
\quad  \vert S(x) \vert \le C_0 \, ,
\\
\label{reSG0} &&  \forall x\,, \quad 2 \re S(x)G(x) +   \D_x S(x) \ge 2
\theta \Id    \, ,
\\
 \label{clS0} &&  S(0) \ge \delta Id  -
C_1 \Gamma^* \Gamma   \, .
\end{eqnarray}
 In \eqref{symS0}, the norm of $S(x)$ is the norm in the space of bounded
operators in $\cH$. Similarly
$S^*(x)$ is the adjoint operator of $S(x)$.
The notation $\re T = \frac{1}{2} (T + T^*)$
is used in  \eqref{reSG0}  for
the real part  of an operator $T$.
When $T$ is unbounded,  the meaning of
$ \re T \ge \lambda $, is that   all $u \in \cD$
  belongs to the domain of $T$ and satisfies
$\re \big( Tu, u \big) \ge   \lambda  \vert   u \vert^2 $ .
Property \eqref{reSG0} has to be understood in this sense.

\begin{lemma}[\cite{MZ3}] \label{le22}
If there is a symmetrizer $S$, then for all
$u \in C^1_0([0, \infty[; \cH ) \cap C^0([0, \infty[; \cD)$, one
has
\begin{equation}
\label{estabstr}
 \theta \Vert u \Vert^2 + \delta \vert u (0) \vert^2
\le \frac{C_0^2}{\theta } \Vert f  \Vert^2 \, + C_1 \vert \Gamma u(0)
\vert^2\, ,
\end{equation}
  where $f := \partial _x u - G u$.  Here, $\vert \, \cdot \, \vert $ is the norm in $\cH$
and $\Vert \, \cdot \, \Vert $ the norm in $L^2([0, \infty[ ; \cH)$.
\end{lemma}

\begin{proof}[Proof from \cite{MZ3}]
Taking the scalar product  of $S u$  with 
\eqref{eqabstr} and integrating over $[0, \infty)$, \eqref{symS0} gives
\begin{equation}
\label{identener}
\begin{aligned}
- ( S(0) & u(0), u(0))  = \int \D_x ( S u, u) dx
	=\int \big( (2   \re S G + \D_x S) u, u \big) dx
 + 2 \re \int \big( S f, u \big) dx\,.
\end{aligned}
\end{equation}
By \eqref{reSG0},
$
\int \big( (2   \re S G + \D_x S) u, u \big) dx  \ge
2  \theta  \Vert  u  \Vert^2
 \, .
$
By \eqref{clS0} and the boundary condition
$ \Gamma u(0)=0$,
$
 ( S(0) u(0), u(0)) \ge \delta \vert u (0) \vert^2  - C_1 \vert \Gamma
u(0) \vert^2
\,.
$
 By \eqref{symS0}
$
 2 \Big\vert  \int \big( S f, u \big) dx \Big\vert \le 2 C_0 \Vert f\Vert
\, \Vert u \Vert \le
	{\frac{C_0^2} { \theta}} \Vert f\Vert^2 + \lambda \Vert u\Vert^2 \,.
$
Thus identity \eqref{identener} implies the energy estimate \eqref{estabstr}.
\end{proof}

\br\label{pwrmk}
The requirement $S\in C^1$ in Lemma \ref{le22} may evidently be relaxed to
\emph{$S$ continuous and piecewise $C^1$}, an observation that is useful
in patching together symmetrizers on different $x$-domains.
For, the resulting new boundary terms in the integration \eqref{identener}
then cancel by continuity.
\er

\subsection{Nonlinear equivalence}\label{s:nlequiv}
In the present setting, restricting for ease of exposition to the simpler, smooth case, 
the relevant family of problems is
\ba\label{nlpres}
 \partial_{x_1} v - G(x_1;\eta,\lambda,  v) v &:=\tilde f:=A_1^{-1}f,\\
 G(x_1;\eta,\lambda,  v)&:=-A_1^{-1}\Big( \lambda  + \sum_{j=2}^d i\xi_j A_j(\bar w+  v)  + E(\bar w)\Big),
 \ea
 with $|\eta,\lambda|$ sufficiently large, $\Re \lambda \geq \lambda_*$, $\lambda_*<-\theta<0$, and $\|v\|_{\hat H^{s_0}}$ sufficiently small, and $\mathcal{H}$ the $L^2$ norm in $x_1$, augmented with a second family
\ba\label{nlspres}
 \partial_{x_1} v - G(x_1;\eta,\lambda,  v)  \partial_{x_1}^s v &:=\tilde f_s,\\
 G(x;\eta,\lambda,  v)&:=-A_1^{-1}\Big( \lambda  + \sum_{j=2}^d i\xi_j A_j(\bar w+  v)  + E_s(\bar w)\Big),
 \ea
 obtained from the the $\partial_{x_1}^s$ differentiated perturbation equations, which differ in form 
 only in the coefficient $E_s$ replacing $E$, involving additional summands of order $|\partial_{x_1}\bar w|$
 and in the forcing $\tilde f_s$ replacing $\tilde f$, consisting of $\partial_{x_1}^s\tilde f$ plus lower derivative
 terms.
 Note that the active dependence on $v$ is only through $x_1$, with $x_2,\dots, x_d$ 
 playing the role of fixed parameters. {\it Thus, this is indeed a family of one-dimensional problems},
 with the difference from the linear version consisting only in dependence on the small unknown function $v$.


 \subsubsection{The quantitative estimates of Lannes}\label{s:quant} 
 At this point, we recall the useful quantitative estimates of Lannes \cite{L}
 bounding pseudodifferential operators and commutators for symbols of limited 
 smoothness; see also \cite{Sr,BS,Met5}.
 The theory of Bony \cite{B} gives estimates with coefficients uniformly bounded so long as the Lipshitz norms 
 of the symbols are uniformly bounded. However, for nonlinear iteration, sharper bounds are often desirable.
 The results of \cite{L} state that for $s$ sufficiently large, commutators, adjoints, etc., 
 associated with paradifferential symbols of total order $s+1$ have $H^s$ operator norm bounded by
 the $H^{s-1}$ norm of the {\it gradient of} the symbol plus the $H^{s-1}$ norm of the symbol.
Here, one may understand the first, gradient term as a bound on the first corrector in the Taylor
expansion of the commutator, and the second, lower-order term as bound on the remainder 
($\rho_1$ in Lannes' notation).
This bound is controlled by but does not control the lower- (than the total $s+1$ of the composition) order 
 $H^s$ norm of the symbol. 
 The latter distinction becomes important in the small-amplitude case discussed in Section \ref{s:smallamp} below.

For further details, see the commutator estimates of \cite[Thm. 7, p. 529]{L}, 
the quantity $H^{t_0}_{ref}$ defined in \cite[Eq. (4.7), p. 519]{L}, and the
quantities $n$ and $N$ on which $H^{t_0}_{ref}$ depends, defined in \cite[Eqs. (2.2) 
and (2.4), pp. 502 and 503]{L};
see also \cite[Thms. 1 and 2, pp. 514 and 515]{L} and \cite[Cor. 30, p. 515]{L} bounding the
action of operators in Sobolev norms in terms of Sobolev norms of their symbols.
For the main, $H^s$ (rather than $H^{s-1}$ of gradient) bounds, 
see also the alternative, more streamlined treatment of Sroczinski in \cite[\S 2]{Sr}, 
focused entirely on Sobolev-based bounds.
For general discussions of pseudo- and paradifferential calculus,
see \cite[Appendix C]{BS} and \cite{Met5}.  

\medskip

 {\it Thus, above,} $s_0 \leq s-1$ is taken sufficiently large that the $H^{s_0}$ norm in $x,t$ controls $L^\infty$
 and Lipschitz norms and also commutator estimates,
 and our goal is to find symmetrizers $S=S(x_1;\eta,\lambda,  v)$ satisfying
 \eqref{symS0} and \eqref{reSG0}. If successful, this yields high-frequency damping estimates for $v$ in 
 $\hat H^s$, hence in $H^s$, for $v$ small enough in $H^{s_0}$. 

 \subsubsection{Reduction to perturbed linear case}\label{s:reduct}
 We discuss only the basic $L^2$ bound as the others go similarly. Denote the desired symmetrizers
 as $S(x_1;\eta,\lambda, v)$, so that the linearized resolvent equation 
 corresponds to the case with $v$ set to zero, i.e., $S=S(x_1;\eta,\lambda, 0)$, $G=G(x_1;\eta,\lambda, 0)$,
replacing $\D_{x_1}$ in \eqref{reSG0} with the partial derivative $\partial_{x_1}$.
That is, treat the $v$ argument entirely as a parameter in a family of modified linear resolvent equations.
 Then, we have the following partial equivalence result analogous to that for the time-local version in \cite{Z2}.

 \begin{proposition}\label{nlprop}
There exist symmetrizers yielding pseudodifferential damping \eqref{pdamp}
for the original problem for $\|v\|_{H^{s_0}}$ 
sufficiently small if and only if there exist symmetrizers for the family of 
modified linear problems $G=G(x_1;\eta,\lambda,v)$ with $\|v\|_{H^{s_0}}$ 
sufficiently small.
 \end{proposition}

 \begin{proof}
	 We first note that we may change $\partial_{x_1}$ back to $\D_{x_1}$ with no harm when
	 $\|v\|_{H^{s_0}}$, hence $\|\partial_{x_1}v\|_{L^\infty}$ is sufficiently small, 
since this just adds an absorbable error term of order 
	 $\|\partial_{x_1}v\|_{L^\infty}\ll 1$.
	 As symbol commutator errors are of order proportional to the $H^{s_0-1}$ norm of $\nabla_{y,t} v$,
	 these also may be absorbed when $\|v\|_{H^{s_0}}$ is sufficiently small, giving 
	 pseudodifferential damping by the same energy estimate as in the proof of Lemma \ref{le22}.
 \end{proof}

 \br\label{prabrmk}
 In the strictly parabolic case, there is no need for quantitative commutator estimates, as favorable terms
 are automatically of higher order. In this case we may set $v_0=0$ in Proposition \ref{nlprop} above.
 \er

\subsection{The smooth small-amplitude case}\label{s:smallamp}
In the smooth small-amplitude case that the entire profile $\bar w(x_1)$ 
lies close to a constant (necessarily equlibrium) solution $w(x_1)\equiv w_0$, with 
$\|\bar w-w_0\|_{W^{1,\infty}}$ sufficiently small, we can make a much 
stronger statement.  For, in this case we may paralinearize in all coordinates 
$x_1,\dots, x_d, t$ to reduce to an entirely symbolic computation, since 
commutator terms, controlled by Lipshitz norm are then absorbable errors 
in the estimate.

Namely, we reduce to the family of problems 
\be\label{sG}
\tilde G v=f, \qquad \tilde 
G= \tilde G(\tilde \eta, \lambda, \tilde v)= 
\lambda + \sum_{j=1}^d i\eta_j A_j(w_0 + \tilde v) -E(w_0+\tilde v),
\ee
with $\tilde v:= (\bar w-w_0) +v$, $\tilde \eta:=(\eta_1,\dots, \eta_d)$,
and we seek energy estimates bounding $\|v\|_{H^s}$ by 
$C(\|f\|_{H^s}+\|v\|_{L^2})$ for $\Re \lambda\geq -\theta$.
Similarly as for the partially linearized system, we can accomplish this for 
bounded $\tilde \eta$ by taking an energy estimate of $\lambda\, \Id v$ 
against the equations for $|\lambda|$ large, whereas it is trivial for 
$|\lambda|$ bounded (so that all frequencies $(\tilde \eta,\lambda)$ are bounded,
and $H^s$ norms therefore bounded by $L^2$ norms times bounded factor 
$C(1+|\tilde \eta,\lambda|^s)$).  
Hence we need only consider the case $|\eta|$ sufficiently large.

In this fully paralinearized case, we thus seek symmetrizers 
$S(\tilde \eta,\lambda,\tilde v)$ satisfying 
\be \label{constreSG0} 2 \re S(\tilde \eta,\lambda,\tilde v)
\tilde G(\tilde \eta,\lambda,\tilde v) \ge 2 \theta \, \Id
\ee
for all $|\tilde v|$ sufficiently small and $\tilde \eta\in \R^d$ sufficiently large.
Note, by the existence theory, that small-amplitude relaxation shocks are always 
smooth; hence, we need not consider the case of subshocks and 
boundary condition $\Gamma$.

By the above discussion, we have immediately the following simplified version of Proposition \ref{nlprop}.

\begin{proposition}\label{constequivprop}
For a fixed equilibrium state $w_0$ and $\|\bar w-w_0\|_{W^{1,\infty}}$ 
sufficiently small, there exist symmetrizers yielding pseudodifferential damping
\eqref{pdamp} for the original problem for $\|v\|_{H^{s_0}}$ sufficiently small 
if and only if there exist symmetrizers $\tilde S(\tilde v_0)$ in the sense of
\eqref{symS0}-\eqref{constreSG0} for the family of modified constant problems 
$\tilde G=\tilde G(\tilde \eta,\lambda,\tilde v_0)$ with $\tilde v_0$ constant
and sufficiently small.
\end{proposition}

Moreover, we can greatly sharpen this result under the following
mild structural assumptions on the system \eqref{blaw}.
Let $\sigma(M)$ denote the spectrum of 
a matrix or linear operator $M$, and 
$$
\tilde T=\tilde T(w,\eta):= \sum_{j=1}^d \eta_j A_j.
$$
Then, we assume for $w$ sufficiently close to $w_0$:

\be\tag{A2}\label{A2}
\hbox{\rm Eq. \eqref{blaw} is {\it hyperbolic},}
\ee
i.e., $\sigma(\tilde T)$ is real and semisimple, and
\be\tag{A3}\label{A3}
\hbox{\rm Eq. \eqref{blaw} is of {\it constant multiplicity}} 
\ee
i.e., $\sigma(\tilde T)$ is constant multiplicity, hence
eigenvalues and associated eigenvectors depend smoothly on $(w,\eta)$ and 
analytically on $\tilde \eta$, with eigenvectors linearly independent 
(cf. \cite[Def. 2.2]{MZ5}).
Under these conditions, we have the following definitive conclusions, reducing
pseudodifferential damping for small-amplitude waves
to high-frequency stability of the linearized dispersion relation about $w_0$.

\begin{theorem}\label{constequivcor}
Assuming \eqref{A2}-\eqref{A3}, for a fixed equilibrium state $w_0$ and $\|\bar w-w_0\|_{W^{1,\infty}}$ 
sufficiently small, there exist symmetrizers yielding pseudodifferential
damping \eqref{pdamp}  for the original problem 
for $\|v\|_{H^{s_0}}$ sufficiently small if and only if there exist symmetrizers in the sense of
\eqref{symS0}-\eqref{constreSG0} for the family of modified constant problems 
$\tilde G=\tilde G(\tilde \eta,\lambda,0)$. This in turn is equivalent to the high-frequency stability condition
\be\label{chf}
\hbox{$\Re \sigma (\tilde T(w_0, \tilde \eta)) \leq -\theta <0$ for 
$\tilde \eta\in \R^d$ and $|\tilde \eta|$ sufficiently large.}
\ee
\end{theorem}

\begin{proof}
First, we note that high-frequency stability for $w(x)\equiv w_0$ is equivalent 
by Parseval's identity to uniform invertibility of 
$(\lambda + i\tilde T(\tilde \eta))$ for $|\tilde \eta,\lambda|$ large,
and $\Re \lambda \geq -\theta$, which implies in particular \eqref{chf}.
On the other hand, under the constant multiplicity assumption \eqref{A3},
there exist uniformly bounded and nonsingular changes of coordinates $R(\eta)$ 
diagonalizing $\tilde A(\tilde \eta):=\sum_{j=1}^d \eta_j A_j$, hence, by standard 
matrix perturbation theory \cite{Kat}, also changes of coordinate $\tilde R$ 
diagonalizing $\tilde T(\eta):= i \tilde A(\tilde \eta)- E$ when $\eta$ is 
sufficiently large, with, moreover, spectra given to order $O(1)$ by the 
pure imaginary spectra of $i A(\eta)$ plus the diagonal 
entries of $RER^{-1}$.\footnote{Diagonal in the case of simple eigenvalues, 
block-diagonal for higher multiplicity eigenvalues.}
It follows that 
$$
\Re(\tilde R \tilde G(\tilde \eta,\lambda){\tilde R}^{-1}=
\Re \big( {\tilde R}(\lambda + \tilde T(\tilde \eta)){\tilde R}^{-1} \big)\geq 
\tilde \theta >0
$$
for $|\tilde \eta|$ sufficiently large and $\Re \lambda \geq -\theta_2 >0$, and thus
$ \Re(\tilde S \tilde G(\tilde \eta,\lambda)) \geq  \theta_2 >0 $ for 
$\tilde S:={\tilde R}^*\tilde R$.
\end{proof}

We remark that the high-frequency stability condition \eqref{chf} is implied by but 
much weaker than the {\it dissipativity condition} of Kawashima-Shizuta
\cite{Ka,KaS,Ze}. Thus,  Theorem \ref{constequivcor} recovers and significantly 
generalizes all previous cases, \cite{KwZ,Kw,MaZ1,YZ}, treated for small 
relaxation waves.

\subsubsection{The symmetric case}\label{s:symmcase}
For Friedrichs symmetrizable systems, $P A(\tilde \eta)$ symmetric for $P$ 
symmetric positive definite, and $\Re P E\geq 0$ with no eigenvector of
$T(\eta)$ lying in $\ker \Re P E$ (the generalized 
``Kawashima genuine coupling condition''), 
\eqref{A3} may be weakened in Theorem \ref{constequivcor} to
\be\tag{A3'}\label{A3'}
\hbox{\rm Eq. \eqref{blaw} is {\it geometrically regular},} 
\ee
i.e., $\sigma(\tilde A)$ and associated eigenvectors depend smoothly on 
$(w,\eta)$ and analytically on $\tilde \eta$, with eigenvectors linearly 
independent (cf. \cite[Def. 2.2]{MZ5}).

For, taking without loss of generality $P=\Id$,
in this case, matrix perturbation theory at an $m$-fold eigenvalue
$\tilde \alpha_j(\tilde \eta)$ of $\tilde A$ gives an associated
diagonal block of $\tilde G(\lambda, \tilde \eta)$
for $|\tilde \eta|$ large of $\lambda + i\alpha \Id_{m\times m}
- \tilde R^* E \tilde R$
plus higher order terms, where $\tilde R_j$ is an $n\times m$ eigenblock of 
orthonormal right eigenvectors of $\tilde T$. But, then
$$
\Re \tilde R_j^* E \tilde R_j= \tilde R_j^* \Re E \tilde R_j= \geq 0
$$
by $\Re E\geq 0$, while $\tilde R_j^* \Re E \tilde R_j\neq 0$ by the 
genuine coupling condition, hence
\be\label{strongchf}
\Re \tilde R_j^* E \tilde R_j= \tilde R_j^* \Re E \tilde R_j= \geq \theta>0
\ee
for some uniform $\theta$, a strengthened version of \eqref{chf}.
This gives again
$$
\Re(\tilde R \tilde G(\tilde \eta,\lambda){\tilde R}^{-1}=
\Re \big( {\tilde R}(\lambda + \tilde T(\tilde \eta)){\tilde R}^{-1} \big)\geq 
\tilde \theta >0
$$
for $|\tilde \eta|$ sufficiently large and $\Re \lambda \geq -\theta_2 >0$, and thus
$ \Re(\tilde S \tilde G(\tilde \eta,\lambda)) \geq  \theta_2 >0 $ for 
$\tilde S:={\tilde R}^*\tilde R$,
similarly as in the final step of the proof of Theorem \ref{constequivcor}. 

This may be compared with the standard ($G$ rather than $\tilde G$)
symmetrizer construction of \cite[Appendix E]{MZ5} 
for hyperbolic-parabolic, or ``viscous'' conservation laws
\be\label{viscous}
\d_t f_0(w) + \sum_{j=1}^d \d_{x_j} f_j(w)=
\sum_{j,k=1}^d \d_{x_j}(B_{j,k}(w)\d_{x_k} w)
\ee
in the {\it low-frequency} regime under analogous symmetry conditions,
and assuming geometric regularity alone.
Recall, for the Fourier symbol $\tilde T(\eta)$, that
high-frequency symmetrizers for \eqref{blaw} have structure 
essentially identical to that for the linearized symbol of 
\eqref{viscous} in the low-freqency regime,
in the case that $B_{jk}=E$ is independent of $j,k$.

\subsection{Partial converse and connection to exponential dichotomies}\label{s:partcon}
The construction of high-frequency symmetrizers can be rather complicated in the large-amplitude
case, and has up to now been done largely ``by hand''; see for example \cite{GMWZ1,GMWZ2,GMWZ3}.
Thus, it is a bit unclear at the outset when or how one could construct such a symmetrizer.
In this section we discuss these issues, pointing out a connection to exponential dichotomies
that both illuminates the original definition of symmetrizers described in Section \ref{s:symm} and gives
useful guidance for their construction.

In the case of classical Kreiss symmetrizers with constant-coefficients, there is a strong
link to Lyapunov theory motivating the introduction of (constant) symmetrizers.
See, for example, the discussion in \cite{MZ3} of intermediate frequencies for which there 
exists a spectral gap for the associated coefficient matrix for the linearized system.
Here, spectral gap is replaced by the variable-coefficient analog of exponential dichotomies,
and Lyapunov theory by the relation of stable and contractive semigroups.

Given a linear system of ODE's
\be\label{gsys}
(d/dx)w=G(x)w 
\ee
defined on the line, with solution operator $\cS(x,y)$,
we recall that an exponential dichotomy is a pair of complementary projections
$P_+(x)$ and $P_-(x)$ such that
\be\label{dcomm}
P_+(x)\cS(x,y)= \cS(x,y)P_+(y), \qquad P_-(x)\cS(x,y)= \cS(x,y)P_-(y)
\ee
for all $x,y$, and
\ba\label{dprop}
|P_+\C_S(x,y)- P_+(y)|&\leq Ce^{-\theta|x-y|},\quad x>y,\\
|P_-\C_S(x,y)- P_-(y)|&\leq Ce^{-\theta|x-y|},\quad x<y.
\ea

Then, we have the following result.

\begin{proposition}[Smooth case]\label{lyapprop}
	For smooth profiles,
	existence of pointwise exponential dichotomies on $(-\infty,\infty)$
	is necessary for
	the high-frequency resolvent bound \eqref{hfres}, which is necessary for existence of
	symmetrizers.  On the other hand, existence of uniformly smooth and bounded global exponential
	dichotomies is sufficient for existence of symmetrizers.
\end{proposition}

\begin{proof} The first assertion follows by Palmer's theorem \cite{P}, the second because existence of
	symmetrizers yields resolvent bound \eqref{hfres}. It remains to prove the third, which is
	the main content of the proposition.
	Consider a system \eqref{gsys} possessing an exponential dichotomy with projectors
$P_{\pm}(x)$. Note that $\range(P_+)(x)$, $\range(P_-)(x)$ uniquely determined as manifolds 
of decaying/growing solutions, $j+k=n$.
Choose bases $\Span\{r_1^+,\dots, r_j^+\}$, $\Span\{r_1^-,\dots, r_k^-\}$ (nonunique),
and define the $n\times n$ coordinate transformation
\be\label{T}
T(x):=\bp r_1^+& \dots & r_j^+&r_1^-& \dots & r_k^-\ep(x) \, .
\ee

Then, making the change of variables $w=Tz$, we obtain a system
\be\label{sys}
z'=\Lambda(x) z,
\qquad \Lambda:=T^{-1}GT- T^{-1}T',
\ee
where the matrix $\Lambda $ is block-diagonal, $\Lambda=\blockdiag\{\Lambda_+, \Lambda_-\}$.
This yields the pair of decoupled systems
\be\label{dec}
z_+'= \Lambda_+z_+, \qquad z_-'= \Lambda_-z_-, 
\ee
where, by the exponential dichotomy property, $z_+$ is exponentially stable in forward direction
and $z_-$ in backward direction.

Using the classical device of Lyapunov, we may thus define quadratic forms $Q_\pm(x)$ and associated norms
$\|\cdot \|_{Q_\pm(x)}$ by 
\be\label{Qdef}
Q(x) := \int_x^{\pm\infty} (\cS^T(x,y) \cS(x,y))_\pm dy,
\ee
where $\cS_\pm (x,y)$ are the respective solution operators for systems \eqref{dec}(i)-(ii), obtaining
convergence by exponential decay of $\cS_\pm(x, \cdot)$ in forward/backward directions.
Noting, for a solution $z+(y)= \cS_+ x(y-x)z_+ x(x)$ of \eqref{dec}(i), that
\ba\label{altquad}
\langle z_+(x),Q_+(x)z_+(x)\rangle &=\int_x^{+\infty} \langle z_+(x),\cS_+^T(y-x)\cS_+(y-x)z_+(x)\rangle dy\\
&=
\int_x^{+\infty} |z_+(y)|^2 dy,
\ea
we find that
\be\label{contract}
(d/dx) \langle z_+(x),Q_+(x)z_+(x)\rangle = (d/dx)\int_x^{+\infty} |z_+(y)|^2 dy= -|z_+(x)|^2.
\ee

Finally, defining 
\be\label{dichsymm}
S(x):= \Big( T^{-1,T}\bp -Q_+&0\\0&  Q_- \ep T\Big(x),
\ee
we find that 
\be\label{lbd}
(d/dx)(1/2)\langle w,Sw\rangle(x)\rangle \geq |v(x)|^2\geq \theta |w(x)|^2, \qquad \theta>0.
\ee
But, also
\ba\label{equality}
	(d/dx)(1/2)\langle w,Sw\rangle(x)\rangle &= \Re \langle w, S(dw/dx)+ (1/2)(dS/dx)w\rangle \\
	&= \langle w, (\Re(SG)+ (1/2)(dS/dx)) w)\rangle,
\ea
whence, combining \eqref{lbd} and \eqref{equality}, we obtain \eqref{reSG0}, verifying the claim.

\br\label{semirmk}
Note that the construction of \eqref{Qdef}-\eqref{contract} is similar to the proof that
a stable $C_0$ semigroup is contractive in an appropriately defined norm \cite{Pa}.
\er
\end{proof} 

Proposition \ref{lyapprop} points out a heuristic connection between resolvent condition \eqref{hfres}
and existence of symmetrizers. At a more useful level, it shows to be {\it sufficient}
for finding symmetrizers the more familiar problem of finding smooth dichotomies, 
which may well be possible by linear tools such as WKB expansion.
We demonstrate below that existence of a uniform smooth symmetrizer in some cases {\it does not}
imply existence of a uniform smooth exponential dichotomy.
In such cases, one must rather look directly for symmetrizers; the evident nonuniqueness of both 
dichotomies and symmetrizers making this both possible and technically involved.

In either case, a WKB-type expansion appears to be a useful first step, analogous to Kreiss' spectral 
decomposition step in construction of classical symmetrizers for constant-coefficient problems.
However, we emphasize that this is indeed, only a first step, and in Kreiss' multi-D analysis it is 
the second step of smoothly choosing symmetrizers that is the main technical difficulty,
with ``glancing points'', or bifurcation from nontrivial Jordan blocks being the main obstruction.
We point to \cite{GMWZ5} and \cite{MZ5} in the constant-coefficient case
for examples in which smooth construction seems to be a major (and still open) obstacle.

\subsubsection{The discontinuous case}\label{s:discdich}
Similarly as in the analogous case of a spectral gap 
in the constant-coefficient case, it is straightforward to extend the
``interior symmetrizer construction'' just performed to one achieving 
also  boundary coercivity, \eqref{clS0}, {\it at least in the case of
a profile containing a single ``subshock'' discontinuity}, 
without loss of generality at the origin, $x_1=0$.

Recall \cite{K,M1,M2,Met3}, even for well-posedness/short-time existence
of the linearized equations, the subshock must be of Lax type, satisfying
the linearized Rankine-Hugoniot conditions
\be\label{RH}
\psi[\lambda F(\bar w) + \sum_{j=2}^d \xi_j F_j(\bar w)]+ [A_1 v]=0,
\ee
where $[\cdot]$ denotes jump across $x=0$ and $\psi$ the infinitesimal shock
location,
from which, projecting onto directions orthogonal to
$[\lambda F(\bar w) + \sum_{j=2}^d \xi_j F_j(\bar w)]$ to eliminate $\psi$, 
one obtains a rank-$(n-1) $ frequency-dependent boundary condition
$$
\Gamma(\lambda, \xi_2,\dots, \xi_d) v=0
$$
on the traces $(v(0^+),v(0^-))$.

By Palmer's theorem \cite{P}, we obtain existence of exponential 
dichotomies on $(-\infty,0)$ and $(0, +\infty)$. The linear 
high-frequency stability estimate \eqref{semigpestdisc}, specialized to 
the homogeneous case $f=0$, $g=0$, then gives $v\equiv 0$ for
$H^s$ solutions $v$ of \eqref{dpres}, from which we obtain the 
(variable-coefficient) {\it Kreiss-Lopatinsky condition}
\be\label{KL}\tag{KL}
\hbox{\rm $\Gamma$ is full rank on $(\Sigma_+(0^-),\Sigma_-(0^+))$},
\ee
where $\Sigma_\pm:= \Range P_\pm$, with $P_\pm$ the projectors onto
exponentially growing (decaying) modes in forward $x_1$.

Consulting to the block-diagonaling coordinates of \eqref{T}-\eqref{dec}, 
we have that \eqref{KL} is equivalent to $\Gamma$ full rank on the 
image under $T$ of the
coordinate blocks associated with $Q_+$ for $x<0$ and $Q_-$ for $x>0$.
Thus, redefining \eqref{dichsymm} as
\ba\label{dichsymmdisc}
S(x)&:= \Big( T^{-1,T}\bp -Q_+&0\\0&  C Q_- \ep T\Big(x),\quad x<0,\\
S(x)&:= \Big( T^{-1,T}\bp C Q_+&0\\0&  - Q_- \ep T\Big(x),\quad x>0
\ea
with $C>0$ sufficiently large, we obtain \eqref{clS0},
exactly as in the treatment by Kreiss of the constant coefficient case
with spectral gap. 
(See first part of proof of Main Theorem 4 \cite[top, p. 292]{K}.)

Combining, we have the following discontinuous version of 
Proposition \ref{lyapprop}.

\begin{proposition}[Discontinuous case]\label{lyappropdisc}
For profiles with a single discontinuity at $x=0$,
existence of pointwise exponential dichotomies on $(-\infty,0)$ and $(0,+\infty)$
together with the Kreiss-Lopatinsky condition are necessary for
the high-frequency resolvent bound \eqref{hfres}, which is necessary for 
existence of symmetrizers.  On the other hand, existence of (globally) 
uniformly smooth and bounded 
exponential dichotomies on on $(-\infty,0)$ and $(0,+\infty)$
together with the Kreiss-Lopatinsky condition is
	sufficient for existence of symmetrizers.
\end{proposition}

\subsubsection{The case of multiple discontinuities}\label{s:multdiscdich}
We mention in passing the situation of a steady profile with multiple
Lax-type subshock discontinuities, which cannot happen in the constant-coefficient 
case, but {\it can} happen for relaxation profiles \cite{JNRYZ,FRYZ}.
For instance, let us consider the case of two shocks, at $x_1$ and $x_2>x_1$.
In this case, the Kreiss-Lopatinski condition \eqref{KL} is replaced
by a more complicated rank condition involving the reduced Rankine-Hugoniot
conditions $\Gamma_j$ for both shocks, $j=1,2$, and the values of $v$
on the growing manifold to the left of $x_1$, the decaying manifold
to the right of $x_2$.
This requires a more detailed computation than the one for single shocks,
involving also dynamics on the intermediate zone between the shocks.
We return to this issue in the more concrete setting of Section \ref{s:onetwo},
in which these intermediate zone dynamics may be diagonalized.

\section{The large-amplitude 1-D noncharacteristic hyperbolic case}\label{s:1D}
To examine further the relation between high-frequency spectral stability, existence of symmetrizers, and
exponential dichotomies, it is useful to start with a simple case in which all three coincide,
linked by the familiar tools of WKB expansion and block-diagonalization in the semiclassical limit.

Restrict now to the 1-D case 
\be\label{1blaw}
\d_t f^0(w) + \d_{x} f^1(w)=r(w),
\ee
under our standard assumptions \eqref{A1}-\eqref{A2} of noncharacteristic 
hyperbolicity, together with geometric regularity, \eqref{A3'}.

Consider the stability of a stationary shock profile $\bar w(x)$ of arbitrary amplitude connecting
endstates $w_\pm$ with the standard decay rates 
\be\label{stddecay}
\hbox{\rm
$|\d_x^k(\bar w- w_\pm)|\leq Ce^{-\theta |x|}$ for $x\gtrless 0$ and $0\leq k\leq K$
}
\ee
following, e.g., from nondegeneracy of rest points $w_\pm$ as equilibria of the associated 
standing-wave ODE.
Then, we have the following definitive result, reducing 
pseudodifferential damping for large-amplitude 1-D waves
to high-frequency stability of the linearized dispersion relations about $w_\pm$.

\begin{theorem}[Smooth case]\label{1dwkb}
For smooth profiles in 1D,
assuming \eqref{A1}--\eqref{A3} near endstates $w_\pm$
and \eqref{A1}--\eqref{A2}, \eqref{A3'} near points on the orbit of 
$\bar w(\cdot)$, high-frequency resolvent estimate \eqref{hfres}, 
pseudodifferential damping \eqref{pdamp}, existence of smooth symmetrizers, 
and existence of uniform smooth exponential dichotomies are all equivalent, 
and, moreover, are equivalent to
	\be\label{cond1}
	\beta_j:=\Big( \ell_j^T (A^0)^{-1} E r_j \Big) (w_\pm)\leq -\theta<0 
	\ee
	for simple eigenvalues $a_j$ of $(A^0)^{-1}A$, with 
	left and right eigenvectors $\ell_j$, $r_j$, $\ell_j^T r_j=1$, and
\be\label{cond1mult}
\sigma \Big( L_j^T (A^0)^{-1} E R_j \Big) (w_\pm)\leq -\theta<0 
\ee
for multiple eigenvalues,
with left and right eigenvector blocks $L_j$, $R_j$, $L_j^TR_j=\Id_{m\times m}$.
\end{theorem}

That is, whenever there hold {\it linear high-frequency} resolvent bounds, as are typically
used to obtain linear decay estimates- typically with loss of derivatives, one has automatically 
{\it nonlinear high-frequency damping estimates} \eqref{pdamp} sufficient to close a nonlinear iteration
and establish nonlinear stability/estimates as well, with {no additional work required.}
Moreover, the verification of high-frequency resolvent bounds \eqref{hfres} is also
straightforward, reducing to the linear algebraic conditions 
\eqref{cond1}-\eqref{cond1mult}.
The latter may be recognized as the high-frequency dissipation conditions
of Kawashima \cite{Ka,Ze}, equivalent to uniform linearized high-frequency 
stability estimates for constant solutions $w\equiv w_\pm$;
for further discussion/extensions, see \cite{FS21} and references therein.

\begin{proof} 
For simplicity, we restrict to the case of simple eigenvalues at $w_\pm$,
the case of multiple eigenvalues being essentially identical.
Uniform smooth exponential dichotomies imply uniform resolvent bounds \eqref{hfres}. 
by variation of constants. We will show under the assumptions \eqref{A1}-\eqref{A3}
that condition \eqref{cond1} implies existence of uniform smooth exponential 
dichotomies, and that uniform resolvent bounds \eqref{hfres} imply \eqref{cond1},
whence the result follows by Propositions \ref{nlprop} and \ref{lyapprop}.

We focus on the most difficult case that $|\Im \lambda|\gg 1$ and $-1/C\leq \Re \lambda\leq C$
for $C$ sufficiently large. The case $\Re \lambda \gg 1$ is similar but straightforward.
Denoting $\lambda= \gamma \pm i/\eps$, $0<\eps\ll 1$, without loss of generality $\lambda=\gamma+i/\eps$,
we express the eigenvalue equation $\lambda A_0 v + A_1\partial_x v-Ev=0$ in semiclassical limit form
\be\label{sc}
	\eps \partial_x v' +  \Big( i A_1^{-1}A_0 + \eps  A_1^{-1}(  E-\gamma A_0) \Big) v =  0,
\ee
or equivalently
\be\label{scequiv}
	\eps \partial_x v' +  \Big( i A_1^{-1}A_0 + \eps  (A_0^{-1}A_1)^{-1}(  A_0^{-1}E-\gamma ) \Big) v =  0,
\ee
preparatory for WKB-type expansion as in \cite{MaZ1,Z4,LWZ,LWZ2}.
Here, $A_j$ and $E$ are functions of $\bar w+ v_0$, $v_0$ small of order $\delta\ll 1$.

By \eqref{A3'}, we have existence of smooth (hence uniformly bounded on
the bounded sets under consideration) eigenvalue and right/left
eigenvector pairs $\alpha_j$, $\ell_j$, $r_j$ for $A_0^{-1}A_1$, $\ell_j^Tr_j\equiv 1$.
Defining $L=(\ell_1^T, \dots, \ell_n^T)^T$ and $R=(r_1,\dots,r_n)$, $LR=\Id$,
and setting $v=Rz$, we may convert \eqref{sc} to the approximately diagonalized form
\be\label{zsc}
	\eps \partial_x z' + \Big(M_0 + \eps (M_1 + N) \Big)z= 0,
\ee
where 
\be\label{M0}
M_0=i \, \diag \{ \alpha_j\}(\bar w),
\ee
\be\label{M1}
	M_1= -\diag \{ \alpha_j^{-1}(\beta_j-\gamma) + O(\delta +  e^{-\theta |x|})\}(\bar w), \ee
and $N= O( \delta  +  e^{-\theta |x|})$ vanishes on its diagonal.
As $M_0 + \eps M_1$ has diagonal entries separated by distance $\sim 1$ near
the endstates $w_\pm$, we may by a further coordinate transformation $\Id+O(\eps)$
reduce the off-diagonal term $N$ to $O(\eps)$ in the limit as
	$x\to\pm \infty$ (see, e.g. \cite{MaZ1})
while leaving the form of $M_0$ and $M_1$ unchanged.

By the assumed convergence \eqref{stddecay}, coefficients
	$\alpha_j^{-1}(\beta_j-\gamma) (\bar w+v_0) $
converge exponentially to an order $|v_0|$ perturbation of their values at $\pm \infty$, 
hence condition \eqref{cond1}
gives for $\gamma\geq -1/C<0$, and some uniform $\theta>0$, $C>0$,
\be\label{uni}
-|\alpha_j^{-1}|(\beta_j-\gamma) \geq \theta + Ce^{-\theta |x|},
\ee
giving exponential dichotomies on the diagonal part of the homogeneous equation
$$
\eps \partial_x z' + \Big(M_0 + \eps M_1  \Big)z= 0
$$
by simple exponentiation of individual (scalar) coordinates, with uniform growth/decay bounds
$Ce^{-\theta |x-y|}$.
	(Here, we have crucially used assumption \eqref{A1}, guaranteeing that the 
sign of eigenvalues $\alpha_j$ do not change as $x$ traverses from 
	$-\infty$ to $+\infty$, in concluding from \eqref{uni} the existence
	of exponential dichotomies on the line.)

Rewriting as
\be\label{pertb}
\partial_x z' + \Big(\eps^{-1} M_0 +  M_1 +O(\eps)  \Big)z= 0,
\ee
we may thus appeal to the tracking/reduction lemma of \cite[\S 3.2]{MaZ5} to conclude
the existence of an $\Id + O(\eps)$ coordinate change converting \eqref{pertb} to
exact diagonal form
\be\label{zsdiag}
\partial_x z' + \Big(\eps^{-1} M_0 +  M_1 +\eps M_2  \Big)z= 0
\ee
with $M_2=O(\eps)$ diagonal, yielding again an exponential dichotomy.

Thus, condition \ref{cond1} implies existence of uniform exponential dichotomies for 
$|\Im \lambda|\gg 1$ and $-1/C\leq \Re \lambda\leq C$, $C\gg 1$.
The case $\Re \lambda \gg 1$ may be treated similarly, redefining $\eps=1/|\lambda|$ and considering
only the principal, order one term $-A_1^{-1}A_0( \lambda/|\lambda|) v$, which has a uniform
spectral gap for $\Re \lambda$ sufficiently large, yielding an ``super''-exponential dichotomy with 
rate $e^{-\theta |x|/\eps}$. See, for example the similar analysis in \cite{Z4}.

As all steps in this construction were smooth, this confirms that \eqref{cond1} implies 
existence of smooth uniform exponential dichotomies, completing the second step of the proof.
For the third and final step of the proof, showing that high-frequency resolvent bounds
\eqref{hfres} imply \eqref{cond1}, we have only to note by standard cutoff arguments
that a necessary condition for uniform high-frequency resolvent bounds for the linearized equations about
the profile is the satisfaction of corresponding high-frequency resolvent bounds for the linearized
equations about the constant solutions $w\equiv w_\pm$ at its endstates.
But, these amount to uniform bounds on the inverse of the Fourier-Laplace symbols
$$
(A_0 \lambda + ik A_1 -E)_\pm
$$
for $k\in \R$, $\Re \lambda\geq -1/C$ and $|\lambda|$, $C$ sufficiently large, 
for which a necessary condition is
\be\label{necstep}
|\lambda +  \sigma A_0^{-1}(ik A_1 -E)_\pm| \geq \theta>0.
\ee
Taking $|k|$ sufficiently large, we find by standard matrix perturbation theory \cite{Kat} that 
$$
\sigma A_0^{-1}(ik A_1 -E)= ik \alpha_j + \beta_j + O(|k|^{-1},
$$
hence, taking $\Im \lambda= ik\alpha_j$ in \eqref{necstep}, we obtain 
$|\Re \lambda+ \beta_j|\geq \theta$ for $\Re \lambda \geq -\theta$, giving $\beta_j\geq \theta$
as claimed.
\end{proof}

Theorem \ref{1dwkb} both simplifies and significantly 
generalizes previous results \cite{MaZ1} and \cite{MaZ2} 
for large-amplitude smooth relaxation waves, which besides dissipativity,
\eqref{cond1}-\eqref{cond1mult}, required, respectively, semilinearity and 
symmetrizability of system \eqref{blaw}.

\br\label{ncrmk}
The ingredient behind the very simple form of the result is the noncharacteristicity condition
\eqref{A1}, which insures that the characteristics $\alpha_j$ do not change sign as we follow an
eigenmode from $-\infty$ to $+\infty$, together with the fact that the sign of spatial decay modes
$\beta_j/\alpha_j$ is determined by the sign of $\alpha_j$.
For more general WKB-type analyses, one typically requires an Evans function condition encoding
uniform transversality of left unstable and right stable manifolds of the eigenvalue equation at $x=0$.
When \eqref{A1} fails, the resolvent ODE becomes singular, and the problem becomes
considerably more subtle, requiring the introduction of additional quite different tools.
See, e.g., \cite{DR22,GR25} in the scalar case, or \cite{RZ2} in the vector
case with singularity of rank one.
\er

\subsection{The discontinuous case}\label{s:disc1D}
For a profile with a single discontinuity at $x=0$,
we obtain by essentially the same construction of exponential dichotomies
as in the proof of Proposition \eqref{1dwkb}, now on half-lines
$(-\infty,0)$ and $(0,+\infty)$, together with Proposition \ref{lyappropdisc}, 
an equally decisive result.

\begin{corollary}[Discontinuous case]\label{1dwkbdisc}
	For profiles in 1D with a single discontinuity (without loss of
	generality at $x=0$),
assuming \eqref{A1}--\eqref{A3} near endstates $w_\pm$
and \eqref{A1}--\eqref{A2}, \eqref{A3'} near points on the orbit of 
$\bar w(\cdot)$, high-frequency resolvent estimate \eqref{hfres}, 
pseudodifferential damping \eqref{pdamp}, 
and existence of smooth symmetrizers are all equivalent, and, moreover
are equivalent to high-frequency dissipativity conditions \eqref{cond1} 
(resp. \eqref{cond1mult} in the case of multiple eigenvalues)
at the endstates $w_\pm$ together with the Kreiss-Lopatinsky condition
	\eqref{KL}.
\end{corollary}

\subsection{Multiple subshocks}\label{s:onetwo}
In the 1-D strictly hyperbolic case we are considering,
for which we may reduce as above to diagonal form, one may
for low-dimensional systems carry out a classical damping
estimate for solutions with two subshocks by the methods used in 
\cite{RZ2}, whenever there holds the high-frequency resolvent estimate
\eqref{hfres}.
For $2\times 2$ systems with two subshocks, the 
same technique can be adapted to yield symmetrizers, and therefore
pseudodifferential damping, whenever there holds \eqref{hfres}.
We demonstrate here the essential part of the argument by treating
the case of a bounded interval $x\in [0,L]$ and diagonal system
\ba\label{diagsys}
u^+_t + a_+ u^+_x= -b^+u^+,\\
u^-_t + a_- u^-_x= -b^-u^-,
\ea
$a^\pm \gtrless 0$, $b^\pm>0$, with boundary conditions
\ba\label{intBC}
u^-(0)&=c^- u^+(0),\\
u^+(L)&=c^+ u^-(L).\\
\ea

This yields eigenvalue equations
\be\label{diageig}
(u^\pm)_x= g^\pm u^\pm, \qquad g^\pm= -(a^+)^{-1}(\lambda + b^+),
\ee
with boundary conditions \eqref{intBC}, for which \eqref{hfres}
implies the Evans-Lopatinsky condition of
nonexistence of solutions for $\Re \lambda \geq 0$ and $|\lambda|$
sufficiently large.
By explicit exponentiation of the eigenvalue problem \eqref{diageig}
and examination of \eqref{intBC}, one finds as in \cite{RZ2} that 
this amounts to 
\be\label{EL}
|c^+c^-|< |v^+(L)|/|v^-(L)|,
\ee
where $v^\pm$ are solutions of \eqref{diageig} with $\lambda=0$,
initialized as $v^\pm(0)=1$.

{\bf Alternative symmetrizer construction.}
In this diagonal setting, we may construct a symmetrizer in a more precise
and simple way, following the idea of \cite{RZ2}, as $S=\diag\{s^+,s^-\}$,
with $s^\pm$ defined by
\be\label{sdef}
2\Re s^\pm g^\pm  + s^\pm _x= \eps,
\ee
giving in the zero-interior dissipation limit $\eps\to 0$ solutons
of 
$$
s^\pm _x= -2\Re s^\pm g^\pm = -2(b^\pm/a^\pm) \, s^\pm,
$$
hence
\be\label{ssoln}
s^\pm= |d^\pm/ v^\pm|^2,
\ee
with $d^\pm$ constant, without loss of generality $d^-=1$, $d^+=d$.

It remains to show strict boundary dissipativity for this $\eps\to 0$ limit,
from which we obtain boundary dissipativity by continuity also for
$\eps>0$ sufficiently small, completing the proof.
It is straightforward to see that boundary dissipativity \eqref{clS0} is equivalent
to {\it maximum dissipativity}: $S(0)$ positive definite on $\kernel \Gamma(0)$
and $S(L)$ negative definite on $\kernel \Gamma(L)$, where 
$$
\Gamma(0)=(-c^-, 1)^T, \quad \Gamma(L)= (1,-c^+)^T,
$$
or
$$
\kernel \Gamma(0)=(1, c^-)^T, \quad \Gamma(L)= (c^+,1)^T,
$$

Taking the neutral limit 
\be\label{neutral}
\Gamma(0)^T S(0) \Gamma(0)=0,
\ee
we find by a similar continuity argument that it is sufficient that
$\Gamma(L)^T S(L) \Gamma(L)<0$,
or 
\be\label{finalcond}
|c_+ d/v^+|^2 - |1/v^-|^2 <0.
\ee
Equality \eqref{neutral} gives $d=c^-$, whence \eqref{finalcond} becomes
$ |c^+ c_-|< |v_+/v_-|$, in agreement with \eqref{EL}.

This computation readily extends to the two-subshock case on the
whole line. More important for us, {\it it will give a useful model
for a key multi-D computation in Section \ref{s:E1} later on.}

\br\label{finitermk}
For the above symmetrizer construction, it is necessary that both
the length of the domain and the real part of $g^\pm$ coefficients
be bounded, as otherwise $\max \{1/|v^\pm|\}$ would grow without bound,
even in the constant-coefficient case.
By contrast, the ``Lyapunov-type'' symmetrizers constructed by the method of
Section \ref{s:partcon} remain always bounded, in the constant-coefficient
case reducing to constants.
The extension to a pair of subshocks at $x_1<x_2$ within a profile
defined on the whole real line may be accomplished
by combining the present construction on the intermediate zone $[x_1,x_2]$
between the shocks with Lyapunov-type constructions on $(-\infty,x_1)$
and $[x_2, +\infty)$, with a matching condition at the boundary points $x_1,x_2$.
Treatment of three or more shocks is an interesting linear algebraic problem.
\er

\section{The multi-D case: Kreiss' glancing vs. Airy-type turning points}\label{s:multi-d_ver}
In the previous sections, we have shown by different means for small-amplitude multi-D waves
and for large-amplitude 1-D waves that pseudodifferential damping is essentially equivalent to dissipativity, 
or spectral stability of the constant solutions $w\equiv w_\pm$ at the endstates $w_\pm$
(at least in the setting of smooth profiles to which we have restricted here).
However, as we now show, the situation is quite different in the large-amplitude multi-D case.
In particular, {\it existence of Kreiss symmetrizers is in this setting distinct from,
and more general than, existence of exponential dichotomies, which, moreover, 
typically does not occur.}
Indeed, though they do exist in the multi-D scalar case, they appear generically not to exist in general, even for $2\times 2$ systems, the simplest vectorial case.

\subsection{Kreiss' glancing}\label{s:Kglancing}
The distinction between symmetrizers and dichotomies
may be seen most easily in the constant-coefficient $E=0$ case
\be\label{Kcase}
\partial_x v= G(\eta, \lambda)v=f, \qquad 
G(\eta, \lambda):= -A_1^{-1}( \lambda A_0 + \sum_{j=2}^{d}A_j i\eta), 
\ee
treated in the classical analysis of Kreiss \cite{K,Met3}.
When the matrix $G$, rescaled by $|\lambda, \eta|$ to bounded size, has a spectral
gap, then an analytic dichotomy (with respect to frequences $\eta$, $\lambda$) 
trivially occurs, yielding smooth symmetrizers by Lyapunov's Lemma.
Construction of uniform smooth symmetrizers
comes down, therefore, to treatment of the limit of vanishing spectral gap, i.e. 
the case of spatial eigenvalues $\mu$ of $G$ with real part approaching zero.
Isolated such eigenvalues are readily dealt with by spectral projection, reducing to the scalar case
\cite{K,Met3}, the essential difficulty lying in {\it glancing points}, or limiting frequencies
at which multiple eigenvalues $\mu_j$ collide at an imaginary value, corresponding under strict
hyperbolicity assumptions to a Jordan block of $G$ of order $s\geq 2$.
For general systems $n\geq 2$ and dimensions $d\geq 2$, {\it such glancing points generically occur}.

For illustration, consider the simplest case $s=2$,\footnote{The only one occurring for ideal gas dynamics
or shallow-water flow \cite{YZ2}.} in dimension $d=2$, for which generically there is
a local smooth curve of glancing points $\lambda=i\tau(\eta)$.
Performing spectral projection isolating the $2\times 2$ Jordan block, 
defining the perturbed frequency $\tilde \lambda:= (\lambda-i\tau(\eta))=\tilde \gamma + i\tilde \tau$, 
and dropping higher order terms yields the model problem 
\be\label{GK}
\hbox{\rm $ w'-G(\lambda)w=f$, with $G(\lambda)=\bp 0 & i/\eps \\ \tilde \gamma + i\tilde \tau & 0\ep$,}
\ee
in the high-frequency limit $\eta= 1/\eps$, $\eps \to 0$, where
\be\label{rest}
\hbox{ $\tilde \tau\in \R$, $\tilde \gamma\in \R^+$, and $|\tilde \lambda| = O(1)$.}
\ee

For this simple problem, a symmetrizer is given by
\be\label{gsymm}
S= \bp 0 & 1+i\eps\gamma \\1-i\eps\gamma & 0\ep,
\ee
for which 
$$
\Re SG= \Re \bp 0 & 1+i\eps\tilde\gamma\\1-i\eps\tilde\gamma &0\ep 
\bp 0 & i/\eps\\ \tilde \gamma + i\tilde \tau& 0\ep= 
\bp \tilde\gamma (1-\eps\tilde\tau) &0\\0&\tilde\gamma \ep \geq \tilde\gamma/2
$$
for $\tilde \gamma$ bounded and  $\eps$ sufficiently small.
For the general case accounting for neglected terms, see, e.g., \cite[\S 4]{K},
\cite{Met3}, \cite[Appendix A]{MZ4}, or \cite[\S 5]{GMWZ6}.

On the other hand \eqref{GK} 
{\it does not possess a smooth exponential dichotomy}, as the
eigenvalues of $G$ split as $\sqrt{\tilde \lambda}$ as $\tilde \lambda \to 0$,
with associated stable/unstable projectors blowing up as $1/\sqrt{\tilde \lambda}$.

\br\label{wavermk}
One may verify, for example, that model problem \eqref{GK} is exactly the scenario that occurs
for the 1-D wave equation
\be\label{1dwv}
A_0=\Id, \quad A_1=\bp 1& 0\\0&-1\ep, \quad A_2=\bp 0&1\\1&0\ep
\ee
near the glancing curve $\lambda_*(\eta)=\pm i\eta$, setting
$ \tilde \gamma + i\tilde \tau:= \lambda - \lambda_* $,
changing coordinates to Jordan form,
and dropping terms or irrelevant order in $\tilde \gamma$, $\tilde \tau$; see, e.g., \cite{Z2,Z3}.
\er

\subsection{Variable coefficients and Airy-type turning points}\label{s:varcoeff}
In the variable-coefficient case, the phenomenon of glancing takes a more complicated form, bringing to the 
forefront ideas introduced in an early {\it turning-point} analysis on detonation stability \cite{Er4} 
of combustion pioneer J.J. Erpenbeck, revisited/retranslated in \cite{LWZ,LWZ2,Z9}.
Namely, whereas Jordan blocks of $G$ occur in the constant-coefficient case for fixed $\eta$ and background
state $w$ and special imaginary values $\lambda=i\tau(\eta,w)$, corresponding to glancing points, 
they occur in the variable-coefficient case on a surface $x_1=x_*(\eta, i\tau)$ for {\it each} 
nearby $\eta$, $\tau$.
In the high-frequency limit, block diagonalization as in \cite{Wa,LWZ2}, corresponding
to lowest order with frozen-coefficient spectral projection, change of coordinates in
$x_1$ and $(\lambda,\xi)$ as in \cite{LWZ2}, and discarding of higher-order terms, 
yields in a neighborhood of the turning point the model semiclassical limit problem
\be\label{turneq}
dw/dx-G(\lambda, x)w=f, \qquad G(\lambda, x)= \bp 0 & i/\eps\\ \tilde \gamma
+ ia(x)/\eps & 0 \ep,
\qquad \eps:=1/|\eta|\to 0
\ee
analogous to \eqref{GK}, where $\tilde \gamma$ and $a$ are bounded, 
$\tilde \gamma \in \R^+$, $a\in \R$, $x$ measures distance from $x_*$, and
\be\label{aprop}
\hbox{\rm
$a(0)=0$, with $a(x)\gtrless 0)>0$ for $x\gtrless 0$.}
\ee


For $a'(0)=1$, the point $x=0$ in \eqref{turneq}, corresponding to the 
special point $x_*$ in the original equations, under the assumptions \eqref{aprop}
can be recognized as an ``Airy-type'' {\it turning point;} cf. \cite{Wa,LWZ}. 
That is, {glancing points in $\lambda$ are replaced for variable coefficients 
by turning points in $x$.}
Note in \eqref{turneq} that $\Im\lambda$ has been absorbed 
in shifts of $a$, $x$, hence $\tilde \tau$ in \eqref{GK} effectively replaced
by $a(x)/\eps$.
An important difference from the constant-coefficient case is that
$ia(x)/\eps$ is more singular than the former term $\tau$, requiring
additional care in the symmetrizer construction.
In particular, the classical construction of Kreiss can
be expected a priori to succeed only in a small neighborhood of $x=0$, where
$|a(x)|/\eps$ is small compared to $1/\eps$.
(The setting of Kreiss, translated to our coordinates, is
$\tau= O(1)$; what is needed for his arguments
is $\tau \ll 1/\eps$ \cite[\S 4]{K}.)

\subsection{Symmetrizers vs. dichotomies I}\label{s:svdI}
The model problem \eqref{turneq} has an evident {\it constant} symmetrizer
\eqref{gsymm} on intervals $x\in [-\delta, \delta]$ for which $|a|\leq 1/2$
giving in this case
\be\label{easysym}
\Re SG= \Re \bp 0 & 1+i\eps\tilde\gamma\\1-i\eps\tilde\gamma &0\ep 
\bp 0 & i/\eps\\ \tilde \gamma + ia(x)/\eps & 0\ep= 
\bp \tilde\gamma (1-a(x)) &0\\0&\tilde\gamma \ep \geq \tilde\gamma/2.
\ee
Indeed, {\it \eqref{easysym} holds not only on $[-\delta, \delta]$,
but on all of $[-\infty, \delta]$} (where $a< 1/2$).

On the other hand, we have the following result indicating
the distinction between smooth dichotomies and symmetrizers in multi-D.

\begin{proposition}\label{nodice}
For the model problem \eqref{turneq}, with $a'(0)=1$ (Airy type), 
there does not exist a smooth family
of exponential dichotomies on $x\in [-\delta,\delta]$, 
uniform as $\tilde \gamma \to 0$, for any $\delta>0$.
\end{proposition}

\begin{proof}
Repeated diagonalization using the tracking/reduction lemma as described in the
proof of Theorem \ref{1dwkb} yields that the flow on
$[-\delta, -C\eps^{2/3}]$ and $[C\eps^{2/3}, \delta]$
for appropriate choice of $C$ is well approximated by the WKB approximation.
Namely, we find that separation of eigenvalues 
\be\label{mudef}
\mu_j= \pm i \eps^{-1}\sqrt{a-i\eps \tilde \gamma} 
\sim i \sqrt{a}/\eps, 
\ee
satisfying
$ |\mu_j| \sim \sqrt{|x|}/\eps$ for $x$ small,
dominates diagonalization error $LR_x\sim a_x/a\sim 1/\sqrt{|x|}$ for
$x$ small and $\eps\ll x$, where
\be\label{LRdef}
R=\bp 1 & 1\\ c& -c \ep
\qquad
	L=R^{-1}=(1/2)\bp 1& 1/c\\ 1 &- 1/c \ep, 
\ee
are left and right diagonalizing transformations for $G$,
with
\be\label{cj}
	c:=(\eps/i)\mu_1= \sqrt{a-i\eps \tilde \gamma} \sim \sqrt{a}.
\ee
Namely, $(error/separation) \sim \eps/|x|$ is small of order $\eps^{1/3}$
under our assumption $\eps^{2/3}\leq |x|\leq \delta$
implying by the repeated diagonalization method of
Levinson, as implemented in  \cite{MaZ1}, that by a further 
coordinate change $\Id + O(\eps^{1/3})$ the diagonalization error may be
reduced to any desired power of $\eps^{1/3}$, in particular much less than 
the order one spectral gap on the side $a(x)>0$.
But this then implies, by the tracking/reduction lemma, that there exists an
	exactly diagonalizing transformation of form $\Id + O(\eps^{1/3})$,
on which the diagonal coordinates experience exponential growth/decay rates 
of order 
\be\label{order}
\Re \mu \sim \pm \Re \sqrt{\tilde \gamma i -x/\eps^2},
\ee
which, for $x\in [-\delta, -C\eps^{2/3}]$, is approximately
$$
\pm \Re \sqrt{|x|/\eps^2}\gtrsim C\eps^{1/3}/\eps= C\eps^{-2/3}\to \infty
$$
as $\eps\to 0$.
Thus, on the $a<0$ side, which, for $x\in [-\delta, -C\eps^{2/3}]$, 
the growing and decaying projectors of any exponential dichotomy must be 
exactly the projections onto diagonal components, else there would be a 
contradiction.

On the other hand, the diagonalizing projectors, in original coordinates,
blow up as $1/\sqrt{|x|}$, which, at $x=C\eps^{2/3}$ is order
$\eps^{-1/3} \to \infty$ as $\eps\to 0$, contradicting
the possibility of a smooth dichotomy (since this would imply
smooth, hence bounded, projectors).
\end{proof}

This shows that {\it existence of smooth symmetrizers and existence of smooth
exponential dichotomies are distinct properties in multi-D}, at least on 
bounded intervals.

\br\label{easygaprmk}
On the $a<0$ side, the spectral gap is of the same order $\sqrt{x}/\eps$
as the spectral separation and so we may simply apply the 
tracking/reduction lemma immediately, without the need for preparation by
repeated diagonalization. The latter is mentioned only to show that the limits
of WKB on either side are similar. 
\er

\br\label{rescalermk}
The rescaling $\tilde x= x/\eps^{2/3}$ is the natural one for the Airy equation 
and its solutions \cite{O,O2,LWZ}, motivating our choice of $|x|\geq \eps^{2/3}$ 
in the result. Note that the growth rate $e^{\pm\eps^{-2/3} x}$ on the
$x<0$ side corresponds to exponential decay in these natural coordinates,
corresponding to the fact that the second-order Airy equation, written
as a first-order system in phase variables
$$
(w_1, \partial_{\tilde x}w_1)= (w_1, \eps^{-1/3} w_2) 
$$
\emph{does} possess an exponential dichotomy in these coordinates.
An alternative approach to the one followed here, below, might be to recover
exponential dichotomy by strategic coordinate changes in both dependent
and independent variables following this intuition.
\er

\subsection{Symmetrizers vs. dichotomies II}\label{s:svdII}
We complete this line of investigation by showing that 
existence of smooth symmetrizers and existence of smooth
exponential dichotomies are distinct properties in multi-D 
{\it also on the whole line}.
Note that Proposition \ref{nodice} already implies nonexistence of 
smooth exponential dichotomies for \eqref{turneq} on the line, 
as existence on the line is a stronger property than existence on an interval.
Thus, the following result is sufficient to prove our claim.

\begin{proposition}\label{yessymm}
Let $a(\cdot)$, $a'(\cdot)$ be uniformly bounded, with a single
zero satisfying \eqref{aprop}, and $a$ uniformly bounded from zero
outside $[-\delta, \delta]$, for some $\delta>0$.
Then, for the model problem \eqref{turneq}, there exists a smooth family
of symmetrizers $S(\tilde \gamma, x, \eps)$ 
on the whole line $x\in \R$ for $\tilde \gamma$ bounded 
and $\eps$ sufficiently small, satisfying for some uniform $\theta>0$,
\be\label{punch}
\Re (SG + S_x)\geq \theta \tilde \gamma>0.
\ee
\end{proposition}

\begin{proof}
We first construct a symmetrizer for $x$ on $[\delta, +\infty]$, for
a fixed $\delta>0$. Note that this is the side of the turning point $x=0$
on which the eigenvalues of the principal ($1/\eps$ order) part are pure 
imaginary, i.e., the ``oscillatory side,'' as compared to the ``exponential
side'' $x<0$ considered in the proof of Proposition \ref{nodice} above.
On this region, consider the symmetrizer 
\be\label{SWKB}
	S_{WKB}:=L^* \bp 1&0\\0&-1\ep L
\ee
suggested by ``frozen-coefficient'' diagonalization (WKB), 
with $L$ given by \eqref{mudef}-\eqref{LRdef}.  Direct computation
	(noting that $a>0$ on this region) yields 
\ba\label{Rsymm}
	S_{WKB}&=(1/2)\bp 0 & 1/c\\1/\bar c& 0\ep\\
	   &=
	  (1/2) \bp 0& 1/\sqrt{ a - i\eps \tilde \gamma}\\
	  1/\sqrt{ a+ i\eps \tilde \gamma}& 0\ep \\
	  & = (1/2\sqrt{a}) \bp 0 & 1+ i\eps \tilde \gamma/2a  \\
	   1- i\eps \tilde \gamma/a  &0\ep + O(\eps \tilde \gamma/2a)^2,
\ea
for $\eps \tilde \gamma/a$ sufficiently small.
Dropping the prefactor $1/2\sqrt{a}$ and the negligible $O(\tilde \eps)^2)$ term, 
we arrive at the simple form
\be\label{Sout}
S_{+}:= \bp 0 & 1+i\eps \tilde \gamma/2|a|\\ 1-i\eps\tilde \gamma/2|a|&0\ep
\ee
reminiscent of \eqref{gsymm}.
This satisfies $\partial_x S_{+}=O(\eps \tilde \gamma)$ for
$a$ uniformly bounded from zero, while
\be\label{S+bd}
	\Re (S_+G)=\bp \tilde \gamma/2 & 0\\0 & \tilde \gamma/2|a| \ep,
\ee
whence  
$$
	\Re (S_{+}G+\partial_x S_{-})\geq \tilde \gamma/2 + O(\eps \tilde \gamma) \geq \theta \tilde \gamma
$$
for $\eps \tilde \gamma$ sufficiently small and $a$ bounded from zero, 
verifying (by assumed boundedness of $\tilde \gamma$) that $S_{+}$ is indeed
a symmetrizer on $[\delta, + \infty]$ for $\eps $ sufficiently small.

For a symmetrizer $S_{-}$ on the lefthand side
$[-\infty, \delta]$, we may choose
\be\label{Sin}
	S_{-}:= \bp 0 & 1+i\eps \tilde \gamma/2a(\delta)|\\ 
	1-i\tilde \gamma/2 a(\delta) &0\ep
\ee
similarly as in \eqref{gsymm}; see the comment below \eqref{easysym}.
As $S_+$ and $S_-$ agree at their mutual boundary $x=\delta$, we have
by Remark \ref{pwrmk} that their concatenation
$$
S:= \begin{cases}
S_-, & x<0,\\
S_+, & x\geq 0
\end{cases}
$$
(since continuous and piecewise $C^1$) is a symmetrizer on the whole line.
\end{proof}

\section{A basic large-amplitude damping result in multi-D}\label{s:basic}
As mentioned above, the model equation \eqref{turneq} with $a'(0)\neq 0$ 
is a canonical form for an Airy-type turning point \cite{LWZ}.
And, indeed, the symmetrizer construction for \eqref{turneq} given
in Proposition \ref{yessymm} may be readily adapted to give a general
damping estimate for smooth, arbitrary-amplitude relaxation shocks in 2-D
under generic hypotheses relevant to applications: in particular to the
multi-D hydraulic shocks studied in \cite{YZ2}.

Let us start with the general setting
\be\label{varKcase}
\partial_x v= G(\eta, \lambda,x;v_0)v=f, \qquad 
G(\eta, \lambda,x):= -A_1^{-1}( \lambda A_0- E + A_2 i\eta), 
\ee
with coefficients $A_j$, $E$ depending on $\bar w(x)+ v_0$, with $v_0$ small in
$H^s$, $A_1$ uniformly invertible.
As in Section \ref{s:1D}, we focus on the critical case 
\ba\label{lamcond}
\lambda&=i/\eps + \gamma, \qquad -\theta\leq \gamma \leq C,\\
\eta &=i\eta_*/\eps + \eta, \qquad |\eta|\leq C,\\
\ea
converting \eqref{varKcase} to semiclassical limit form
\be\label{varKsemi}
G(\eta, \gamma,x;v_0, \eps,\eta_*)=
-\eps^{-1}i A_1^{-1}(  A_0 + A_2 (\eta_*+ \eps \eta)) -A_1^{-1}( \gamma A_0- E),
\ee
where all coefficients and constants are real.
Assuming high-frequency resolvent bound \eqref{hfres} for the associated
linearized problem, we seek to construct uniform smooth symmetrizers $S$,
$$
\Re(SG+\partial_xS)\geq \theta_2>0,
$$
thereby obtaining uniform resolvent bounds. (And, indeed, for purpose of a
damping estimate, we could fix once and for all $\gamma=-\theta$, though
this does not simplify much).
Noncritical frequencies $\lambda= (\gamma_*+ i)/\eps$ yield exponential
dichotomies by standard estimates, hence by Proposition \ref{lyapprop} 
may be ignored. 

Next, we assume that for each choice of critical frequency $\lambda$,
the eigenvalues of the principal part 
\be\label{prince}
-\eps^{-1}i A_1^{-1}(  A_0 + A_2 (\eta_*+ \eps \eta))
\ee
of $G$ split into groups of one or two eigenvalues, none of which intersect
the other groups as $x$ traverses the real line, and which cannot be split
further, i.e., the eigenvalues in a group of two cross either on the real line
or asymptotically as $x\to \pm \infty$.  Moreover, among the latter, we assume
that they cross precisely once, at an Airy-type turning point if somewhere on
the real line, or an asymptotic Jordan block if at $\pm \infty$.

This is the scenario that has been shown to hold in \cite{YZ2} for smooth
(or discontinuous) multi-D
hydraulic shock waves of the inviscid Saint Venant equations, a $3\times 3$
relaxation system of interest in hydraulic engineering.
A similar scenario was shown to hold for (discontinuous) 
ZND detonations in \cite{Er4,LWZ}.
Thus, it is a situation of physical relevance, though not the most general possible.

Using the tracking/reduction lemma as in the 1-D case, we may separate off each
of these groups and treat them individually.
Going a bit further, we notice that groups of a single eigenvalue fit the 1-D
scenario of Section \ref{s:1D}, hence admit exponential dichotomies yielding
symmetrizers by Proposition \ref{lyapprop}. 
These also may therefore be ignored, leaving us with the problem of
treating $2\times 2$ systems of form \eqref{varKsemi}: in effect, reducing
to the $2\times 2$ case.

We now note another, technically simplifying property that is satisfied
for both Saint Venant shocks and ZND detonations, namely,
that the upper righthand entry of the $2\times 2$
principal part \eqref{prince} of the reduced equation be nonvanishing on the 
whole line.  Next, we observe that we may assume without loss of generality that
the principal part be {\it traceless}, since we may subract off the trace
times the identity matrix without affecting either the nonvanishing
off-diagonal hypothesis or the existence of symmetrizers, since
$\Re (SG)= \Re (S(G-{\rm Trace}  G)$ follows from the fact that $G$ is pure imaginary
in each entry.
The following observation, proved by direct computation, shows then, that
a further change of coordinate yields principal part of form 
\be\label{redprince}
G_{principal}= (i/\eps) b \bp 0 & 1\\ a & 0\ep, 
\ee
with all matrix entries real: that is, essentially the scenario of
model problem \eqref{turneq}.
(Though the prefactor $b$ varies with $x$, this does not affect existence of
symmetrizers such as constructed in previous subsections, for which $S_x$
is a neglible error.)

\begin{lemma}\label{tracelem}
	For a traceless matrix $g=\bp d&b\\c&-d\ep$, the choice 
	$L=\bp 1 & 0\\ d/b & 1\ep$, $R=L^{-1}$
	yields
	\be\label{conc}
	LgR- L\partial_x R=\bp 0 & b\\ a & 0\ep,
	\qquad a:=d^2/b^2 - \partial_x (d/b)/b.
	\ee
\end{lemma}

Thus, in both of the two previously investigated cases (Saint Venant and ZND),
the coefficient $G$ of the reduced $2\times 2$ blocks of the
critical eigenvalue problem may be reduced to the 
form \eqref{turneq}, modulo an {\it order one real} summand, coming from the
(reduced, coordinate-changed version of) the secondary part
$-A_1^{-1}( \gamma A_0- E)$ of $G$.

More generally, we consider any system \eqref{varKsemi} for which
\eqref{varKsemi} may be decomposed (by any means)
into a collection of distinct single- and double-eigenvalue blocks of the
principal $1/\eps$ order part of $G$, 
with the latter of form
\be\label{2block}
G(x,\nu, \gamma, \eps)= (ib(x,\nu,\eps)/\eps)\bp 0 & 1) \\  a(x,\nu,\eps) & 0\ep
+ M(x,\gamma, \eps),
\ee
where all coefficients are real and depend smoothly on their arguments,
$b$ is uniformly bounded away from zero, and $a(\cdot, \nu,\gamma,\eps)$
vanishes either at precisely one point $x_*$ in $\R$ or asymptotically 
at $+\infty$ or $-\infty$ but not both, the former of Airy type 
($\partial a/\partial x)|_{x=x_*}\neq 0$.

Fixing hereafter $\gamma=\gamma_*<0$, we assume in addition that 
the limiting matrices $G_\pm$ of as $x\to \pm \infty$ of $G$ have a 
uniform spectral gap, with the ``consistent splitting'' property \cite{AGJ} 
that the dimensions of stable and unstable subspaces of $G_\pm$ are equal; 
moreover, we assume the further condition that consistent splitting hold 
also {\it individually}, on each of the decoupled single-eigenvalue blocks.
Regarding two-eigenvalue blocks, we assume at each Airy point $x_0$ or 
$x=\pm \infty$ the ``block structure condition'' \cite{K,M1,M2,Met3,Met4} 
\be\label{bstruct}
\sgn \, M_{21}= \sgn a'(x_*)b(x_*) \neq 0.
\ee
Similarly as for model problem \eqref{turneq}, this implies that outside a 
small neighborhood of $x_*$ (resp. $\pm \infty$), the block may be further
split into two single-eigenvalue blocks, one growing and one decaying,
for $x$ sufficiently near to $x_*$ (resp. $\pm \infty$).
We assume the augmented consistent splitting condition that these have the same
signs just to the left (resp. right) of $x_*$ as at $-\infty$ (resp. $+\infty$).

\begin{theorem}[Smooth case]\label{basicthm}
Under the assumptions just above, there exists a smooth family of
Kreiss symmetrizers for \eqref{varKsemi} with uniform constant $\tilde \gamma>0$
for $0<\eps\leq \eps_0$ sufficiently small.
\end{theorem}

\begin{proof}
We first dispose of the single eigenvalue blocks. By assumption, these correspond
to distinct eigenvalues 
$$
	\mu_p=(i/\eps)\pm \sqrt{ ab}
$$
of the principal part $(i/\eps)\bp 0 & b\\a & 0\ep$, with $ab$ uniformly bounded
from zero, hence are either pure real, or pure imaginary, throughout the domain
$\R$. In the first case, there is a uniform spectral gap and a symmetrizer
trivially exists. In the second the eigenvalue $\mu$ of the full system including
order one parts has real part $O(1)$, with consistent exponential growth (resp. decay) assumed at $\pm \infty$, whence the mode possesses a trivial exponential
dichotomy by direct exponentiation, and a symmetrizer again exists,
by Proposition \ref{lyapprop}.

We now consider a single $2\times 2$ block \eqref{2block} with a finite Airy point
$x_*$ satisfying \eqref{bstruct}, without loss of generality 
$b,a>0$ and $M_{21}(x_*)>0$, and $x_*=0$.
This is quite similar to the model case \eqref{turneq} treated
above, the two main differences being (i) the presence of new $O(1)$ error terms
$ \tilde M=\bp M_{11}& M_{12}\\0 & M_{22}\ep, $ and (ii) the fact that
$\sgn \, M_{21}$ is specified only at the turning point $x_*$.

Let $\delta>0$ be such that $\alpha:=\max_{[-\delta,+\delta]}|a|$
is sufficiently small and $\beta:= \min_{[-\delta,+\delta]} b$ and
$\beta:= \min_{[-\delta,+\delta]} M_{21}$ are bounded away from zero.  
Define
\be\label{Salpha}
 S_\alpha=\bp 0 & 1+ i\eps \tilde \gamma/2\alpha \\ 
1- i\eps \tilde \gamma/2\alpha&0 \ep.
\ee
similarly as in \eqref{S+bd}.
Then, on the ``Airy region'' $[-\delta, +\delta]$, where $a<\alpha$ 
	\ba\label{Salphaest}
\Re (S_\alpha G)
&=b \bp \tilde \gamma(1- a/2\alpha & 0\\0 & \tilde \gamma/2\alpha \ep
+ \Re S_\alpha \tilde M\\
&\geq b \bp \tilde \gamma/2 & 0\\0 & \tilde \gamma/2\alpha \ep
		+ \bp 0 & O(1)\\O(1)& O(1)\ep\\
&\geq b \bp \tilde \gamma/4 & 0\\0 & \tilde \gamma/4\alpha \ep
\geq \beta \tilde \gamma/4 
\ea
for $\alpha$ sufficiently small, while $\partial_x S_\alpha=O(\eps)$,
hence $S_\alpha$ is a uniform symmetrizer on $[-\delta, +\delta]$.

Let us now treat the ``exponential region'' $(-\infty,-\delta/2]$,
	partly overlapping with $[-\delta, \delta]$,
where $a$ is uniformly negative. On this region, the principal part of
$G$ has a uniform spectral gap of order $1/\eps$, hence possesses a trivial
``frozen-coefficient'', or ``WKB'' symmetrizer $S_-$, with $S, \partial_x S =O(1)$ 
and $\Re(S_- G)\geq \theta/\eps>0$.
Moreover, this symmetrizer is exactly $bS_{WKB}$, where $S_{WKB}$
is as computed in \eqref{Rsymm}, whence, like $\Re S_\alpha \tilde M$
and $S_\alpha$, both $\Re S_-G$ and $\partial_x S_-$ vanish on their
1-1 entries.

Setting $\chi(x)$ to be a smooth transition function, $\chi(-\delta)=0$
and $\chi(-\delta/2)=1$, define the interpolating symmetrizer
\be\label{partS}
	S_{part}(x,\dots):= \chi(x) S_-(x,\dots)+ (1-\chi(x))S_\alpha(x,\dots)
\ee
by partition of unity.
Then,
\ba\label{partrel}
\Re( S_{part}G) + \partial_x S_{part}&=
\chi \Big( \Re( S_{-}G)+ \partial_x S_-\Big) + (1-\chi)\Big( \Re( S_{\alpha}G) +  
\partial_x S_{\alpha}\Big)
+ \chi'\Big( S_\alpha- S_- \Big)\\
&\geq 
b \bp \tilde \gamma/2 & 0\\0 & \tilde \gamma/2\alpha \ep
		+ \bp 0 & O(1)\\O(1)& O(1)\ep
\geq \beta \tilde \gamma/4. 
\ea
Thus, $S_{part}$ is a smooth uniform symmetrizer on $(-\infty,+\delta]$.

There remains only to extend $S_{part}$ to the ``oscillatory regime'' 
$[+\delta,+\infty)$ where $a>0$ and the eigenvalues of the principal
part $G_{p}$ are pure imaginary.
On this side, we use the consistent splitting hypothesis to see that
there is an exponential dichotomy on $[+\delta, +\infty)$ associated with
the WKB splitting, whence there exists a symmetrizer $S_+$ with uniform 
constant $\tilde \gamma_2$ determined by the rates of growth/decay at $+\infty$.
Moreover, the construction of Proposition \ref{lyapprop} allows the freedom
to initiate $S_+$ at point $x=+\delta$, or any other chosen point,
with the ``frozen-coefficient'' symmetrizer $S_{WKB}$ of \eqref{Rsymm}, 
which, on the ``oscillatory boundary'' $x=+\delta$ is within 
$O(\tilde \gamma\eps)^2$ of $S_{a(\delta)}$, where $a(\delta)$
for $\delta$ sufficiently small is in turn (by Taylor expansion)
within $\delta^2=O(\alpha^2)$ of $\alpha$.
It follows that we may alter $S_\alpha$ slightly to exactly match $S_+$
at $x=\delta$ while introducing only absorbably errors
to estimates of the previous cases. Concatenating $S_{part}$ with $S_+$
we thus obtain a continuous uniform symmetrizer on all of $\R$,
by Remark \ref{pwrmk}.

The case of a $2\times 2$ block with Airy point at $\pm \infty$,
without loss of generality $+\infty$,
may be treated similarly as, but more simply than, the case of a finite point,
with the treatment of the ``Airy region'' $[1/\delta, +\infty)$ reducing to
essentially the constant-coefficient case of Kreiss \cite{K}.
\end{proof}

\brs\label{krmks}
1. The reduction to traceless form \eqref{conc} may be regarded as a variant
of \cite[Lemma 4.3]{K} in the special case of a Jordan block of dimension
$s=2$, designed to highlight the connection with Airy's equation and the
analyses of \cite{Er4,LWZ,YZ2}.
As can be seen by the details of the argument, form \eqref{turneq} is not 
needed on the whole line, but only in a vicinity of the turning point $x_*$,
allowing treatment of more general cases with further care.

2. The crucial construction of the family of symmetrizers $S_\alpha$ likewise
reflects the key observation of \cite[Lemma 4.4]{K} for Jordan blocks of
arbitrary dimension $s\geq 2$ that there is a degree of freedom in the choice
of symmetrizers allowing arbitrarily large constants in coordinates $2, \dots, s$,
allowing absorbtion of remaining error terms.

3.  The block structure condition for Airy point at infinity, though phrased here
as an additional assumption, is in the hyperbolic constant multiplicity case
a consequence of high-frequency dissipativity
of the end states, which is implied by our standard assumption of 
high-frequency resolvent bound \eqref{hfres}.  
See \cite[Lemma 2.7, Eqn. (2.21)]{K} in the stictly hyperbolic case, \cite{Met4} 
in the general constant-multiplicity case.

4. Similarly, augmented consistent splitting at finite Airy points may be shown to
follow from high-frequency resolvent assumption \eqref{hfres}, via 
the high-frequency spectral stability it implies.
We believe, but have not shown, that, similarly, 
the block structure assumption for finite Airy points is forced by
\eqref{hfres}. For the analytic-coefficient case, it is a consequence
of the spectral analysis of \cite{LWZ}, via exact conjugation to the Airy equation.
Analytic approximation of coefficients together with continuity of resolvent
bounds seems a promising direction for a proof in the general case, bypassing
spectral stability to appeal directly to \eqref{hfres}.

5. We note that spectral properties of the Airy operator and perturbations
are encoded in the construction of Kreiss symmetrizers, without the need
for exact conjugation.
In particular, the case of Airy points at infinity admits considerable simplification
by this approach, cf. \cite{LWZ}.

6. The above hypotheses (some redundant!) are verified by the direct WKB expansion
of \cite{YZ2} in the case of smooth 2-D hydraulic shock profiles
for the inclined Saint Venant equations.
\ers

\subsection{Discontinuous profiles: the two cases of Erpenbeck}\label{s:gendisc}
We now turn to the discontinuous case, assuming the same structural assumptions
as in the smooth case-- i.e., block-diagonalizability of 
\eqref{varKsemi} for critical frequencies $\lambda=i/\eps + \gamma$ 
into scalar and $2\times 2$ blocks, each containing at most one Airy-type turning
point, without loss of generality at $x_1=0$--
but now allowing also a single shock discontinuity at $x_1=\tilde x$,
without loss of generality lying in $[-\infty,0]$.
The case that there is a shock but no turning point is straightforward
by the methods of Section \ref{s:discdich}.
We therefore focus on the case that {\it both shock and turning point occur}.

As observed by Erpenbeck in the detonation setting \cite{Er4}, there
are two main cases in the simultaneous presence of turning point
and discontinuity: what we will call the {\it oscillatory intermediate region}
and {\it exponential intermediate region} cases, depending whether the
eigenvalues of \eqref{redprince} are {\it pure imaginary} or {\it real}
on the interval $(0,\tilde x)$ between turning point and shock.
In the exponential case, the associated Kreiss-Lopatinsky condition corresponding
to spectral stability reduces in the high-frequency limit just to entropy
admissibility of the shock, hence gives no additional information other than
dissipativity of the endstates $w_\pm$ \cite{LWZ,Z9}.
In the oscillatory case, dynamics on $[0,\tilde x]$ play a role in spectral
stability as well, similarly as in the analysis of Section \ref{s:onetwo}
There are also the boundary cases when turning point and shock coincide, or
the turning point lies at $-\infty$.

We discuss each of these cases in turn, describing briefly
how the symmetrizer constructions of
the previous subsection may be modified to accomodate the shock.

\subsubsection{Oscillatory intermediate region}\label{s:E1}
The most delicate case is the oscillatory one, for which the
high-frequency Kreiss-Lopatinski encodes information that must be
incorporated in the symmetrizer construction.
On $(-\infty, \delta)$, $0<\delta \ll 1$, we may use the symmetrizer construction
of the smooth case, after the coordinate transformation $x\to -x$.
Recall that in diagonalized coordinates, this amounts to the choice
$S=\diag\{-1,1\}$ at $x=\delta$.
Choosing a symmetrizer on $[\delta, \tilde x]$ matching this value at $x=\delta$
then amounts to solving the bounded interval case treated in Section \ref{s:onetwo},
with artifical left boundary condition $u^+(0)=u^-(0)$. 

But, likewise, the WKB study of \cite{LWZ,YZ2} 
(encoding Airy dynamics as described in \cite{LWZ,Z9}, taking the decaying
solution at $-\infty$ to the sum of growing and decaying solutions at $0^+$)
gives a Kreiss-Lopatinsky condition corresponding to the condition \eqref{KL}
for the bounded interval. Thus, again, the high-frequency resolvent condition
implies \eqref{KL}, and thereby existence of a smooth symmetrizer.
To the right of the subshock, 
there is no turning
point, and so uniform exponential dichotomies hold. Thus, we may treat
this part
by a straighforward weighting as in Section \ref{s:discdich}.

\subsubsection{Exponential intermediate region}\label{s:E2}
In the exponential case, one may construct a symmetrizer on $(-\infty, \delta]$
as in the smooth case, then join to essentially arbitrarily weighted
frozen-coefficient symmetrizers constructed by WKB, by the same interpolation
argument of the previous subsection, weighting the decaying mode sufficiently 
strongly, as in the proof of Proposition \ref{lyappropdisc}.
The symmetrizer on $[\tilde x-\infty, +\infty)$ may then be
chosen by a straighforward weighting as in Section \ref{s:discdich}.

\subsubsection{Turning point at infinity}\label{s:Ebdry}
The case of a turning point near $-\infty$, with shock location taken
without loss of generality at $x=0$, reduces essentially to the bounded
domain computation of the previous subcases.
For, as noted in the treatment of smooth profiles, the same constant symmetrizer
used for finite turning points applies also near infinity, up to the value
$x=-M$, $M\gg 1$ such that $a(-M)=a_0$, $a_0>0$ small and fixed.
Using the same matching arguments as previously, we reduce again to the 
bounded domain problem of Section \ref{s:onetwo} on $[-M,0]$
with artifical left boundary condition $u^+(0)=u^-(0)$,
and the argument goes as before.

\subsubsection{Coinciding subshock and turning point}\label{s:Ebdry}
In the case that the turning point coincide (or nearly coincide), 
one must modify the choice of symmetrizer to emphasize the weight of
the decaying mode, in order to achieve boundary dissipativity, \eqref{clS0}.
The same issue arises in the constant-coefficient case treated by Kreiss, 
and can be handled following the approach of [Lemmas 4.1-4.4]{K}.

Namely, substituting for \eqref{Salpha} the modified symmetrizer
\be\label{Salphad}
 S^d_\alpha=\bp 0 & 1+ i \eps \tilde \gamma/2\alpha \\ 
1- i \eps \tilde \gamma/2\alpha& d \ep.
\ee
for $d>0$ sufficiently large and $\alpha>0$ sufficiently small,
we find that $S_\alpha$ satisfies both interior and boundary dissipativity
on a small neighborhood $[\tilde x-\delta,\tilde x]$ of the turning point and shock.
Note that this is exactly analogous to the choice in \cite{K} of
symmetric principal part $D$ and corrector $iF$, $F$ skew, given by
\be\label{D}
D=\bp 0 & 1\\ 1 & d\ep, \qquad F=\bp 0 & C\\-C& 0\ep,
\ee
with $d$, $C$ sufficiently large, in the simplest setting of an
$s\times s$ Jordan block with $s=2$.

It remains to extend this symmetrizer to $(-\infty, -\delta]$.
Using the interpolation argument of the smooth exponential case,
we see that we need only construct a symmetrizer on $(-\infty,\tilde x -\delta]$
matching $S_\alpha^d$ to order $o(1)$ on $(\tilde x-2\delta, \tilde x)$.
But this is readily achieved by choosing an appropriate member of
the one-parameter family of WKB symmetrizers 
\be\label{SWKBtheta}
	S_{WKB}:=L^* \bp \theta &0\\0&-1\ep L
\ee
generalizing \eqref{SWKB}.
For, as noted in \cite[Lemma 4.1]{K}, there is a unique one-parameter
family of symmetrizers $D$ of the principal Jordan block part
$$
\bp 0 & 1\\ 0 & 0\ep
$$
of $G$ at the turning point, whence the principal parts of 
the families described by \eqref{D} and \eqref{SWKBtheta}
must agree.
Extension of the symmetrizer to $[\tilde x,+\infty)$ is
straightforward, as in previous cases.


\subsubsection{Basic discontinuous damping theorem}\label{s:Ethm}
Collecting the above arguments, we obtain the following basic result,
sufficient to treat multi-D stability in 
(some) physically interesting situations.

\begin{theorem}[Discontinuous case]\label{basicthmdisc}
For profiles with a single discontinuity,
	under the structural assumptions of Theorem \ref{basicthm} (in particular,
	at most a single Airy-type turning point in each separate block), 
together with linear high-frequency resolvent bound \eqref{hfres} for the 
associated linearized problem, 
there exists a smooth family of Kreiss symmetrizers for \eqref{varKsemi} with 
uniform constant $\tilde \gamma>0$
for $0<\eps\leq \eps_0$ sufficiently small.
\end{theorem}

\br\label{svrmk}
As noted in \cite{FRYZ,YZ2}, the hypotheses of Theorem \ref{basicthmdisc}
are satisfied in the oscillatory indermediate region case 
for all discontinuous monotone decreasing profiles of the 
Saint Venant equations for inclined shallow-water flow, 
which occur precisely in the hydrodynamically stable case of Froude number $F<2$.
Nonmonotone discontinuous profiles, occurring for $F>2$, 
fall in the exponential intermediate region case, but do not satisfy
the dissipitivity condition and endstates, since this is equivalent
to hydrodynamic stability.
It would be very interesting to establish 
convective nonlinear stability as done in the 1-D case in \cite{FRYZ},
i.e., stability in an appropriately exponentially weighted norm,
as observed numerically in \cite{YZ2}.
\er

\br\label{nonKreiss}
Our treatment of the oscillatory case is not a simple
analog of the constant-coefficient Kreiss Theory \cite{K}, 
but a blend of that theory with the rather different set of ideas described in 
Section \ref{s:onetwo}, originating in the 1-D investigations of Rodrigues, 
Duch\^ene, Gareaux, Faye, and the author, as described in \cite{RZ2},
with motivations also from the spectral investigations of \cite{Er4,LWZ,Z9,YZ2}.
\er

\section{Nonlinear iteration}\label{s:nests}
We now describe how the a priori estimates \eqref{gpars}-\eqref{gparcor} for finitely-supported 
functions in time may be converted to a general estimate like \eqref{idamp} sufficient to close a nonlinear iteration.
As in previous sections, we confine our discussion for simplicity to the case of smooth solutions.

\subsection{Short-time theory}\label{s:short_time}
For smooth waves, we may paralinearize in $x_1,\dots, x_d$ and not $t$, accepting $O(1)$ commutator errors
on the order of the Lipschitz norm of $v$, to obtain
\be\label{shortloc}
\partial_t v-H(\nu,v)=f, \qquad H:= \sum_{j=1}^d A_j(\bar w +  v)i\nu_j -E(\bar w).
\ee

Assuming \eqref{A2}-\eqref{A3}, we can as for the small-amplitude case make a uniformly invertible 
coordinate change $R(\eta,v)$ diagonalizing $\sum_{j=1}^d A_j(\bar w +  v)i\nu_j $, and the entries
of the resulting diagonal matrix are real.
Taking the inner product of $(1+|\nu|^{2s})v$ against equation \eqref{shortloc}
then gives readily (in transformed coordinates)
\be\label{shortest}
\partial_t \|v\|_{H^s}^2 \leq C \big(\|v\|_{H^s}^2+\|f\|_{H^s}^2),
\ee
hence, by Gronwall inequality,
\be\label{intshortest}
\|v(t)\|_{H^s}^2 \leq e^{Ct}\|v(0)\|_{H^s}^2+  C\int_0^t e^{C(t-\tau)} \|f(\tau)\|_{H^s}^2 \, d\tau
\ee
for all $t\geq 0$, so long as the Lipshitz norm of $v$ remains bounded,
a result holding also in the original coordinates.

By a standard fixed-point argument, we then obtain for the original system \eqref{blaw} the following
short-time existence/continuation result.

\begin{proposition}\label{pshort}
	Assuming \eqref{A2}-\eqref{A3},
	for initial data $w_0=\bar w+v_0$, with $v_0 \in H^{s}$, $s$ sufficiently large,
	there exists for some $T>0$ a unique solution of \eqref{blaw} on $0\leq t\leq T$, 
	such that for $v(x,t):= w(x,t)-\bar w(x_1)$,
	$$
	v \in L^\infty([0,T]; H^s)\cap C^0 ([0,T]; H^{s-1/2}),
	$$
	with $\|v(t)\|_{H^s}$ upper semicontinuous.
	Moreover, the solution continues so long as the Lipshitz norm of $v$ remains bounded and
	the $L^\infty$ norm of $v$ remains sufficiently small.
\end{proposition}

\subsection{Linear bounds}\label{s:newlinbd}
Proposition \ref{pshort} seems interesting as a frequency-dependent version 
of Friedrichs symmetric hyperbolic theory. 
In the linear setting, it yields global existence with bounded exponential
growth, furnishing a starting point from which we may verify existence of
a Laplace transform, in order to obtain the following sharpened bounds.
Consider the linear homogeneous version 
\be \label{linhomv}
(\partial_t -L)v=0
\ee
of \eqref{shortloc}.

\begin{proposition}\label{sharpprop}
Assume along with \eqref{A2}-\eqref{A3} that the linearized operator $L$
have no spectra with positive real part.
	Then, assuming the high-frequency resolvent estimate \eqref{hfres}
(as follows from a linear damping estimate), the solution of \eqref{linhomv} 
satisfies for any $\theta>0$, and some $C=C(\theta>0$,
\be\label{sharplin}
	\int_0^{+\infty}e^{-2\theta s} \|v(t)\|_{H^s}^2 ds
	\leq C \|v(0)\|_{H^s}^2.
\ee
\end{proposition} 

\begin{proof}
	Consider $\tilde v:= \chi(t) v$, where
	$\chi$ is a $C^\infty$ cutoff function vanishing at $t=0$ and $1$
	for $t\geq 1$.
	Note that $\tilde v$ satisfies 
	$
	(\partial_t-L)\tilde v= \chi'(t)\tilde v,
	$
	with $\chi'$ supported on $0\leq t\leq 1$, and $\|v(t)\|_{H^s}\leq
	C\|v(0)\|_{H^s}$ for $0\leq t\leq 1$.
	By Parseval's inequality, we thus obtain
\eqref{sharplin} for all $\theta$ for which
	the inverse Laplace transform formula holds in $L^2$ sense in time,
	in particular for $\Re \lambda \geq C$, with $C>0$ sufficiently large,
	as in \eqref{intshortest}.
	Using the resolvent identity as in standard semigroup theory,
	together with the assumed uniform high-frequency estimate \eqref{hfres},
	and the assumed boundedness of the resolvent on compact frequency sets
	with $\Re \lambda>0$,
	we may move the contour to $\Re \lambda=\theta$ for any $\theta>0$,
	so long as $\|Lv(0)\|\leq C$, yielding
	the result for $v(0)\in D(L)$. The result for general $v(0)$
	then follows by a density argument, using boundedness/continuity with
	respect to initial data of the linearized solution operator. 
	For smooth solutions, one may alternatively apply Pr\"uss theorem
	\cite{Pr} to obtain the same result.
\end{proof}

Equally important for our purposes is the underlying estimate \eqref{intshortest},
and the resulting bounds for solutions defined on $[0,T]$ 
and any finite $\tau>0$, of
\be\label{mainone}
\int_0^\tau \|v(t)\|_{H^s}^2 dt \leq C \|v(0)\|_{H^s}^2 + C \int_0^\tau \|f(t)\|_{H^s}^2 dt
\ee
and
\be\label{maintwo}
\int_{T-\tau}^T \|v(t)\|_{H^s}^2 dt \leq 
 C \int_{T-2\tau}^{T-\tau} \|v(t)\|_{H^s}^2 dt
+ C \int_{T-2\tau}^T \|f(t)\|_{H^s}^2 dt .
\ee

\subsection{Truncation to finite time} \label{s:truncated}
We are now ready to establish our main, finite-time result.

\begin{theorem}\label{smooththm}
Assuming \eqref{A2}-\eqref{A3}, for a solution $w=\bar w +v$ of \eqref{pres} 
	on $[0,T]$, for which (i) there exist symmetrizers yielding
\eqref{pdamp}, (ii) the linearized operator $L$ about $\bar w$ has no spectra 
	with positive real part, and (iii) $v$ is sufficiently small in $H^s([0,T]\times \R^d)$, 
there holds for some $\eta>0$ sufficiently small the nonlinear damping estimate \eqref{idamp}.
\end{theorem}

 \begin{proof}
	 First, notice that we may easily reduce to the case that $f$ is
	 supported on $ t\leq T$, by multiplying $f$ by a $C^\infty$ cutoff
	 $\chi(t)$ equal to $1$ for $t=T-\tau$ and $0$ for $t=T$.
	 Denoting the resulting modified solution as $\tilde v$ and
applying \eqref{gparcor}, we obtain
	 $$
\int_0^T e^{2 \gamma (T-t}\|\tilde v(t)\|_{H^s}^2 dt 
\leq C_2 \int_0^T e^{2 \gamma (T-t)}\big(\|f(t)\|_{H^s}^2+ \|\tilde v(t)\|_{L^2} \big) dt,
$$
and thus
\be\label{tkey}
\int_\tau ^{T-\tau} e^{2 \gamma (T-t}\|v(t)\|_{H^s}^2 dt 
\leq C_2 \int_0^T e^{2 \gamma (T-t)}\big(\|f(t)\|_{H^s}^2+ \|\tilde v(t)\|_{L^2} \big) dt.
\ee
Adding to this equation the equation \eqref{maintwo}, bounding the term
$ C \int_{T-2\tau}^{T-\tau} \|v(t)\|_{H^s}^2 dt$ on the righthand side 
of \eqref{maintwo} using again \eqref{tkey}, and rearranging, we obtain the result,
{assuming that it holds for $C^\infty$ $f$ supported on $ t\leq T$.}

Next, consider $C^\infty$ $f$ supported on $0\leq t\leq T$, and $v$ vanishing
at $t=0$, and introduce the modified function $v^m$ defined as follows.  
For $0\leq t\leq T$, take $v^m=v$. For $t\geq T$, let $v^m$ satisfy the homogeneous
linear equation about profile $\bar w_x$. For $T\leq t\leq T+1$, let
$v^m$ satisfy a homogeneous equation obtained by $C^\infty$ interpolation between
$\bar w$ and $\bar w+v$ in the arguments of the coefficients of \eqref{pres}.
Finally, set 
\be \label{vdag}
	 v^\dagger:= (1+ \chi(t)(e^{-\tilde \theta (t-T)}-1) v^m, 
\ee
where $\chi$ is a $C^\infty$ cutoff function vanishing for $0\leq t\leq T$
and identically one for $t\geq T+1$. Then, we find that $v^\dagger$
	 satisfies \eqref{pres}, with modified coeffients as described above,
	 and with modified forcing $f^\dagger$ defined as $f$ for $0\leq t\leq T$,
	 $ -\tilde \theta v^\dagger$ for $t\geq T+1$, 
	 and bounded for $T\leq t\leq T+1$ by
\be\label{chibds}
\xi'(e^{-\tilde \theta (t-T)}-1) + \tilde \theta \chi(e^{-\tilde \theta (t-T}-1)=
O(\tilde \theta)\|v^m(t)\|_{H^s} = O(\tilde \theta \|v(T)\|_{H^s},
\ee
where the final inequality follows for $ \|v(T)\|_{H^s}$ sufficiently small
by short-time existence theory applied on the time interval $T\leq t\leq T+1$.

By Proposition \ref{nlprop}, there exist symmetrizers yielding 
\eqref{pdamp} for this interpolant as well.
Moreover, by linear bounds \eqref{sharplin}, the Laplace transform of
$ v^\dagger$ is well-defined for $\Re \lambda \geq \gamma=-\eta$ for $\eta >0$
sufficiently small relative to $\tilde \theta$.
Applying now \eqref{gpars}, we thus obtain
\be\label{finaldag}
\int_0^{+\infty} e^{2 \gamma (T-t}\|v^\dagger (t)\|_{H^s}^2 dt 
\leq C_2 \int_0^{+\infty} e^{2 \gamma (T-t)}\big(\|f^\dagger(t)\|_{H^s}^2
+ \|v^\dagger (t)\|_{L^2} \big) dt.
\ee
By \eqref{chibds}, using again \eqref{maintwo}, the contribution of $f^\dagger$ 
from $T\leq t\leq T+1$ may be absorbed in the lefthand side 
for $\tilde \theta$ sufficiently small.
For $t\geq T+1$ on the other hand, 
$
\|f^\dagger\|_{H^s} = \tilde \theta \|v^\dagger\|_{H^s},
$
and so can again be absorbed on the lefthand side of \eqref{finaldag} for 
$\tilde \theta$ sufficiently small.

Finally, by linear estimate \eqref{sharplin},
the contribution of $\|v^\dagger (t)\|^2$ for $t\geq T+1$ may be bounded
by $C\|\dagger v(T+1)\|_{L^2}^2$, while by short-time existence theory
the contribution from $T\leq t\leq T+1$ may be bounded by $C\|v(T)\|_{L^2}^2$,
each of which by \eqref{maintwo} may be bounded by
$C\int_{T-1}^{T} \|v(t)\|_{L^2}^2 dt$, hence absorbed in the $L^2$
contribution for $0\leq t\leq T$ on the righthand side of \eqref{finaldag}.

Removing all of these absorbable terms reduces the righthand side of
\eqref{finaldag} to an integral from $0$ to $T$.  Discarding the terms 
for $t\geq T$ on the lefthand side and multiplying by $e^{2\gamma T}$
thus gives 
$$
\int_0^{T} e^{2 \gamma (T-t}\|v^\dagger (t)\|_{H^s}^2 dt 
\leq C_2 \int_0^{+\infty} e^{2 \gamma (T-t)}\big(\|f^\dagger(t)\|_{H^s}^2
+ \|v^\dagger (t)\|_{L^2} \big) dt,
$$
establishing \eqref{gparcor} for the key case of $C^\infty $ $f$ 
supported on $t\geq 0$ and $v$ vanishing at initial time $t=0$.
From \eqref{gparcor}, we readily obtain \eqref{idamp} from the 
bound 
$$
\|v(T)\|_{H^s}^2\leq C \int_{T-1}^T \|v(t)\|_{H^s}^2 \, dt
$$
similar to \eqref{maintwo}, established by short time bound
$\|v(T)\|_{H^s}\leq C \int_{T-1}^T \|v(t)\|_{H^s}^2\, dt$
together with Jensen's inequality.

The case of nonvanishing initial data or forcing is now straightforward, 
introducing the modified function $v\sharp:= \chi(t)v$ with $\chi$
a smooth cutoff function that is identically zero for $t\leq 0$ and identically
one for $t\geq 1$, and using \eqref{mainone}.
This satisfies \eqref{pres} with $C^\infty$
forcing $f^\sharp:= \chi f + \chi' v$ vanishing for $t\leq 0$ and $v^\sharp$
vanishing at $t=0$, agreeing with $f$ and $v$ for $t\geq 1$. Estimates for
$0\leq t\leq 1$ may be obtained from short-time theory, completing the proof.
\end{proof}

{\bf Conclusion:} Theorem \ref{smooththm} completes our goal for smooth waves, recovering the central
estimate in the classical damping case, sufficient to close a nonlinear iteration and establish stability,
{\it provided there exist symmetrizers yielding a pseudodifferential damping estimate \eqref{pdamp}.}

\br\label{altprmk}
Alternatively, the theorem may be phrased without reference to symmetrizers by
assuming there exist damping estimates \eqref{pdamp} for both the paralinearized
problem and the modified problem obtained as described above
by interpolating coefficient arguments toward the linearized problem over
the time interval $T\leq t\leq T+1$.
\er

\br\label{difrmk}
As for the symmetrizer estimates earlier on, the step of truncation requires
additional care in the damping case compared to classical arguments for short-time
well-posedness, in which latter case error terms need only be shown to be order one
and not small in order to be absorbed.
In particular, we seem to require the additional assumption (ii) that the linearized
operator $L$ about the (exact) traveling wave have no spectra of positive real part.
It is an interesting open question whether this is a technical artifact or indeed 
necessary for the result. Examination of the linear setting suggests that
presence of essential spectra with positive real part may indeed present a problem, 
an issue connected with causality in the inverse Laplace transform formula.
However, in any case condition (ii) is no restriction for the use in nonlinear stability arguments
for which the damping estimates are intended, since absence of positive real part spectra of $L$ 
is a necessary condition for the linearized stability estimates on which such arguments
are based.
\er

\subsection{Absorbing nonlinear sources $f$ and $g$}\label{s:absorb}
In the smooth multi-d case, it is straightforward to absorb nonlinear contributions
already at the level of symmetrizer estimates, yielding a nonlinear damping bound
\be\label{smoothmd}
 \|v(T)\|_{H^s_\alpha}^2\leq Ce^{-\eta T} \|v(0)\|_{H^s_\alpha}^2
 + C\int_0^T e^{-\eta (T-t)} \|v(t)\|_{L^2_\alpha}^2 \, dt
 \ee
corresponding to that in \cite{Z1,Z2}, and sufficient to close the estimates
as done in these references. In the 1-D case, term
$\psi_t \bar w_x$ in $f$ (and, in the discontinous case,
multiples of $\psi_t$ in $g$) leads to the modified bound
\be\label{smooth1d}
 \|v(T)\|_{H^s_\alpha}^2\leq Ce^{-\eta T} \|v(0)\|_{H^s_\alpha}^2
 + C\int_0^T e^{-\eta (T-t)}\Big( \|v(t)\|_{L^2_\alpha}^2+ |\psi_t(t)|^2\Big) \, dt
 \ee
corresponding to the nonlinear damping estimate obtained in \cite{MaZ1,MaZ2,YZ}
by Kawashima-type energy methods, and is again sufficient to close the 
nonlinear argument as done in the smooth case in \cite{MaZ1,MaZ2} and in
the discontinuous case in \cite{YZ}.

In the discontinuous multi-D case, terms 
$\nabla_{x_2,\dots,x_d,t}\psi \bar w_x$ in $f$ lead, similarly, to 
$$
 \|v(T)\|_{H^s_\alpha}^2\leq Ce^{-\eta T} \|v(0)\|_{H^s_\alpha}^2
 + C\int_0^T e^{-\eta (T-t)}\Big( \|v(t)\|_{L^2_\alpha}^2+ 
 |\nabla_{x_2,\dots,x_d,t}\psi (t)|^2\Big) \, dt,
$$
which should be sufficient to close a nonlinear iteration
by arguments like those of \cite{YZ}.
However, this last step is an aspect of the discontinuous multi-D case (only)
that has yet to be carried out.

\section{Applications}\label{s:appl}
Theorem \ref{smooththm}, when combined with the various methods developed in Sections
\ref{s:nlinequiv}--\ref{s:basic}
gives immediately a number of new results in various domains, along with simplification of previous ones.

\subsection{Extensions and simplification}\label{s:Lapplications}
Our results yield immediately a number of simplifications and extensions for existing work.
For example, In the 1-D large-amplitude relaxation case, they recover and somewhat
simplify the damping result of \cite{MaZ2}, removing contortions associated with translating
linear algebraic estimators to Friedrichs symmetrizer form.
Likewise, in the viscous shock case \cite{Z1,Z2}, global multi-D stability 
is known only for hyperbolic-parabolic systems of ``Kawashima type'' (roughly, simultaneously symmetrizable)
at the endstates of the shock. Yet, high-frequency bounds \eqref{hfres} have been shown in the different context
of bounded-time vanishing viscosity already in \cite{GMWZ2,GMWZ3} by 
Kreiss symmetrizer techniques under the much weaker, apparently sharp assumption \cite[(H5)]{GMWZ2,GMWZ3}
that the endstates be stable as constant solutions, a condition on the dispersion relation of the
associated Fourier symbol. As simultaneous symmetrizability was used in \cite{Z1,Z2} only to establish a
nonlinear damping estimate, we thus extend immediately the multi-D long-time stability results of \cite{Z1,Z2}
to the larger class of systems considered in \cite{GMWZ2,GMWZ3}, putting the two theories (long-time vs. small-viscosity) finally on par.

\subsection{The small-amplitude case}\label{s:applications}
Theorem \ref{constequivcor} both simplifies and extends a number of small-amplitude results,
replacing symmetrizability assumptions with the constant-coefficient high frequency damping condition 
\eqref{chf}, which is implied by but more general than the standard dissipativity conditions for
the endstates of the wave, themselves implied by but more general than the symmetry-based Kawashima 
conditions of \cite{KaS,Ka,Ze}.  In particular, this extends the results of \cite{Kw,KwZ} on stability of 
smooth multi-D relaxation profiles from symmetrizable to merely high-frequency dissipative systems, {\it assuming spectral stability}.
See \cite{SrZ} for a 
substantial further extension of these methods to the case of
{second-order hyperbolic relaxation systems} arising in relativistic gas dynamics,
yielding for the first time
multi-D stability of smooth relativistic viscous shock profiles.

In the case considered in \cite{YZ2} of {\it channel flow}, one may verify spectral
stability as well, to obtain a complete result independent of further assumptions.
For \cite{YZ2}, the Fourier frequencies $\xi$ now are selected from a discrete set
$\xi_j=2\pi j/L$, $j\in \ZZ$,
where $L$ is the width of the channel. For $|\xi_j|$ large, one obtains
spectral stability from Theorem \ref{constequivcor}.  For $|\xi_j|$ bounded and
nonzero, one obtains spectral stability by continuity from the constant solution,
which enjoys a spectral gap.
For $\xi_j=0$, the spectral problem reduces to 1-D, hence stability follows
by the 1-D result of \cite{PZ}.
This completes the proof of spectral stability.
Nonlinear stability can then be shown by the essentially 1-D 
argument for flow in a channel developed in \cite{TZ,Z5}, 
using exponential slaving of nonzero $\xi_j$ modes to the $\xi_j=0$ mode,
together with the nonlinear damping estimate afforded by 
Theorem \ref{constequivcor}.  

\subsection{The large-amplitude multi-D case}\label{s:potential}
As regards our main motivating example, of large-amplitude multi-dimensional shocks in the 
Saint Venant model, the preliminary WKB analysis has been already done in \cite[\S 2.4, p. 32]{YZ2} 
for the linear problem, yielding the high-frequency linearized spectral stability
needed for our analysis, and verifying the assumptions made in Theorem \ref{basicthm}.
We note further, as pointed out in \cite{YZ2}, that Airy points {\it do appear}
for waves of all amplitudes, both smooth and discontinuous type,
so that the analysis of Theorem \ref{basicthm} is indeed needed to obtain nonlinear damping.
Combining these ingredients with the linearized resolvent bounds obtained in \cite{Kw}
for arbitrary-amplitude smooth multi-D profiles of general relaxation systems, 
one obtains the result 
that {\it spectral stability} in the sense of a standard Evans function
condition \cite{Kw,YZ2} {\it implies linearized and nonlinear stability} with standard decay rates 
matching those of the parabolic case \cite{Z1,Z2,Z3}.
Evans stability has been demonstrated for all such waves by 
nonrigorous numerics in \cite{YZ2}. 
It would be very interesting to convert this to rigorous computer-aided proof.

\section{Discussion and open problems}\label{s:disc}
Our analysis resolves, first, a long-standing question of the relation 
between high-frequency stability \eqref{hfres}, damping \eqref{damp}, and nonlinear iteration.
As noted in \cite{Z2}, \eqref{damp} is {not} implied by \eqref{hfres}; hence
\big(\eqref{hfres} $\Leftrightarrow$ \eqref{pdamp} $\Leftrightarrow$ \eqref{gparcor} \big)
$\Leftarrow$ \eqref{damp}, with \eqref{gparcor} sufficient for nonlinear iteration.
Second, when an energy estimate \eqref{damp} does not hold or is not readily available, it allows us to bring to bear 
frequency-dependent pseudodifferential energy estimates and the full Kreiss symmetrizer machinery to establish
\eqref{hfres}, as was \eqref{semigpest} in the classical shock case 
\cite{K,M1,M2,Met3}, resolving the bottleneck described in the introduction: 
most significantly for smooth large-amplitude waves in multi-D.

\subsection{Open problems}\label{s:open}
A further issue is the complete treatment of nonlinear stability
of multi-D discontinuous ``subshock''-containing solutions, generalizing
1-D results of \cite{YZ,FRYZ}.
Here, we have presented without giving full details the main ingredient 
for nonlinear damping of construction of symmetrizers for
critical frequencies $\lambda=i/\eps + \gamma$.
As mentioned at the end of Section \ref{s:absorb}, another important
step, not appearing in the smooth multi-D theory, is the integration
in a full nonlinear stability argument of the $\nabla_{x,t} \psi$
terms appearing in the nonlinear source terms $f$ and $g$.
Moreover, the requisite low-frequency linear bounds,
though accessible to the same inverse Laplace-transform techniques used in 
the smooth case, have up to now not been carried out: in particular
the estimation of $\nabla_{x,t} \psi$ terms.
We expect but have not verified that these can be treated in a similar 
way as in the 1-D analysis of \cite{YZ}.
{\it We regard this}, along with systemization/generalization of the
rather special arguments presented here,
{\it as the next main open problem in the theory}. 

In short, beyond the various applications described in Section \ref{s:appl},
the main contribution of the present work in our view is to take the theory to 
the threshold of a Kreiss symmetrizer-based treatment of nonlinear damping
and full nonlinear stability of large-amplitude {\it discontinuous} profiles 
in multi-D, under the assumptions of noncharacteristic hyperbolicity \eqref{A1} 
and constant multiplity \eqref{A2}.
Where constant multiplicity fails, e.g., some cases in MHD and anisotropic 
wave motion, further analysis is needed, even in the constant-coefficient case; 
see related discussion in \cite{MZ5,GMWZ5}.

Likewise, failure of noncharacteristicity is an important direction for 
further study, with applications to periodic roll wave solutions \cite{JNRYZ,YZ2}.
For, as shown in \cite{RZ2}, failure of condition \eqref{A1} does not imply 
failure of damping, nor linear equivalence, but only of the general arguments 
developed above.  Given the important example of Saint Venant roll 
waves \cite{JNRYZ,RZ2}, it would be very interesting
to investigate this case further using other techniques.

A further interesting problem, as described in Remark \ref{finitermk}
is the treatment of three or more subshocks or systems of size $n>3$
in the 1-D setting, and associated linear algebraic issues.
We note that the classical symmetrizers constructed for roll waves in 
\cite{RZ2} are ``sharp'', in the sense that they exist whenever there
holds the high-frequency resolvent bound \eqref{hfres}, for
systems of size $n\leq 5$, but apparently not in general.
It would be very interesting, therefore, to reconsider this problem
from the symmetrizer point of view, a related linear algebraic question,
with the hope of obtaining a sharp result also for systems of size $n>5$.

The extension to waves with multiple Airy-type turning points
is another important problem for applications,
as is the treatment of degenerate turning points of parabolic cylinder-
where $x$ is replaced in \eqref{turneq} by $x^2$ in the lower lefthand entry of $G$- or higher degeneracy type, or of Jordan blocks with $s>2$,
the latter presumably be treated by arguments closer to the original 
framing of \cite{K} (see Remarks \ref{krmks}).
Neither case occurs for the motivating example of shallow water flow \cite{YZ2}.

Finally, even for smooth multi-D relaxation profiles 
(on the whole space as opposed to a channel as discussed above),
there remains the important open question of 
verifying multi-D spectral stability analytically in the small-amplitude limit,
without resort to numerical approximation.
As numerical Evans function computations become quite stiff
in the small-amplitude limit, this is of practical as well as theoretical
concern.
See \cite{PZ,FS02,FS10,Ba} for models of such analyses in related situations.



\end{document}